\documentclass[11pt]{article}
\usepackage[sectionbib]{natbib}
\usepackage{array,epsfig,fancyheadings,rotating}
\usepackage[colorlinks,linkcolor=blue,anchorcolor=blue,citecolor=blue,CJKbookmarks=True]{hyperref}
\usepackage{sectsty, secdot}
\usepackage{amsbsy}
\usepackage{amsmath}
\usepackage{amssymb}
\usepackage{epsfig}
\usepackage{multicol}
\usepackage{multirow}
\usepackage{color}
\usepackage{bm}
\usepackage{threeparttable}
\usepackage{float}
\usepackage{array}
\usepackage{amsmath}
\usepackage{amssymb}
\usepackage{amsfonts}
\usepackage{multirow}
\usepackage{amsthm}
\usepackage{graphicx}
\usepackage{bm}
\usepackage{enumerate}
\usepackage{subfigure}
\usepackage{booktabs}
\usepackage{pdflscape}
\usepackage{url}
\usepackage{xcolor,graphicx,float}
\usepackage{rotating}
\usepackage{times}
\usepackage{indentfirst}
\usepackage{blindtext}
\usepackage{algorithm,algorithmicx,float}
\usepackage{lipsum}
\usepackage[noend]{algpseudocode}
\usepackage{setspace}
\usepackage{booktabs}

\usepackage[sectionbib]{natbib}
\usepackage{array,epsfig,fancyheadings,rotating}
\usepackage[]{hyperref}  
\usepackage{sectsty, secdot}
\usepackage{authblk}
\usepackage{amsmath}
\usepackage{amssymb}
\usepackage{amsfonts}
\usepackage{multirow}
\usepackage{amsthm}
\usepackage{graphicx}
\usepackage{bm}
\usepackage{enumerate}
\usepackage{subfigure}
\usepackage{booktabs}
\usepackage{pdflscape}
\usepackage{url}
\usepackage{xcolor,graphicx,float}
\usepackage{rotating}
\usepackage{times}
\usepackage{indentfirst}
\usepackage{blindtext}
\usepackage{algorithm,algorithmicx,float}
\usepackage{lipsum}
\usepackage[noend]{algpseudocode}
\usepackage{setspace}
\usepackage{booktabs}


\makeatletter
\renewcommand{\@thesubfigure}{\hskip\subfiglabelskip}
\makeatother

\renewcommand{\baselinestretch} {1.5}
\makeatletter 
\def\singlespace{\def\baselinestretch{1}\@normalsize}

\newlength\savewidth

\renewcommand{\theequation} {\arabic{section}.\arabic{equation}}
\@addtoreset{equation}{section}
\newtheorem{theorem}{Theorem}

\newtheorem{lemma}{Lemma}
\newcommand{\bg}{\mathbf{g}}
\newcommand{\bv}{\mathbf{v}}
\newcommand{\bu}{\mathbf{u}}
\newcommand{\bw}{\mathbf{w}}
\newcommand{\ba}{\mathbf{a}}
\newcommand{\be}{\mathbf{e}}
\newcommand{\bA}{\mathbf{A}}

\newcommand{\bV}{\mathbf{V}}
\newcommand{\bI}{\mathbf{I}}
\newcommand{\bX}{\mathbf{X}}
\newcommand{\bS}{\mathbf{S}}
\newcommand{\bs}{\mathbf{s}}
\newcommand{\bt}{\mathbf{t}}

\newcommand{\bZ}{\mathbf{Z}}

\newcommand{\bY}{\mathbf{Y}}

\newcommand{\bG}{\mathbf{G}}
\newcommand{\bH}{\mathbf{H}}
\newcommand{\bq}{\mathbf{q}}
\newcommand{\bP}{\mathbf{P}}

\newcommand{\bh}{\mathbf{h}}
\newcommand{\bz}{\mathbf{z}}
\newcommand{\bc}{\mathbf{c}}

\newcommand{\bbf}{\mathbf{f}}
\newcommand{\bbb}{\mathbf{b}}
\newcommand{\bxi}{\boldsymbol{\xi}}

\newcommand{\R}{\mathbb{R}}

\newcommand{\F}{\mathrm{F}}
\newcommand{\op}{\mathrm{op}}

\newcommand{\tr}{\operatorname{tr}}

\usepackage{stackengine}

\DeclareMathOperator*{\argmin}{argmin}

\makeatletter
\renewcommand{\@thesubfigure}{\hskip\subfiglabelskip}
\makeatother

\makeatletter
\def\singlespace{\def\baselinestretch{1}\@normalsize}

\date{\today}

\title{Optimal rates for aggregation of affine estimators}

\author{Jingfu Peng}
\affil{Yau Mathematical Sciences Center, Tsinghua University}

\begin{document}
\begin{sloppypar}

\maketitle

\begin{abstract}

Aggregation of estimation procedures has found important applications in econometrics, statistics, and machine learning. Classical statistical aggregation theory has mainly focused on a \emph{pure aggregation} setting, where the candidate estimators are either deterministic or constructed using a held-out sample independent of that used for aggregation. When the candidate estimators and the aggregation weights are estimated from the same dataset without sample splitting, Bellec [\emph{Ann. Statist.} \textbf{46} (2018), 30--59] established the optimal rate for model-selection aggregation of affine estimators. Beyond this regime, however, fundamental questions regarding the optimal aggregation rates and the construction of aggregation rules attaining these rates remain largely unresolved. 

In this paper, we consider the problem of aggregating a finite collection of affine estimators to learn an optimal convex combination of them. This framework encompasses a rich class of estimators widely used in statistics and machine learning, including least squares estimators, kernel ridge estimators, random feature regression estimators, and among many others. We establish the minimax rate for convex aggregation of affine estimators. In particular, we show that estimating the weights by minimizing a Mallows' $C_p$ criterion attains the optimal rate. We further study the linear aggregation regime with unrestricted weights and establish matching minimax lower and upper bounds over suitable classes of affine estimators.

\end{abstract}

\bigskip
\baselineskip=18pt

\section{Introduction}\label{sec:intro}

Aggregation of estimation procedures has played a fundamental role in the development of several areas of data science. In econometrics, forecast combination can be traced back to the seminal work of \cite{Bates01121969}. In machine learning, a series of influential methods developed by \cite{Breiman1996Bagging, Breiman1996Stacked, Breiman2001Forest} introduced effective ways to combine multiple predictive models and laid important foundations for modern ensemble learning. These aggregation approaches have achieved remarkable empirical success across a wide range of applications, from recommender systems, as famously demonstrated by the Netflix Prize \citep[see, e.g.,][]{Feuerverger2012}, to presidential election forecasting, weather prediction, and large-scale machine learning competitions. Therefore, understanding the fundamental statistical limits of aggregation and constructing optimal aggregation strategies that attain these limits is of both theoretical and practical interest.  

\subsection{Motivation}

In statistical learning theory, a rigorous theoretical framework for aggregation was developed through a series of seminal works by \cite{Nemirovski2000}, \cite{Tsybakov2003}, \cite{Catoni2004}, and \cite{yang2004aggregating}, which classified aggregation problems into three basic regimes: model-selection (MS) aggregation, convex aggregation, and linear aggregation. The objective of MS aggregation is to mimic the best individual estimator among a given collection of candidates by aggregation, whereas convex and linear aggregation aim to mimic the best convex and linear combinations of the candidate estimators, respectively. In the regression setting, the minimax rates for these three aggregation problems were established \citep[see, e.g.,][]{Tsybakov2003}, characterizing the optimal worst-case excess risk attainable by any aggregation procedure relative to the minimum risk under the corresponding weight constraints. Building on this minimax perspective, substantial advances have been made in developing aggregation procedures that are optimal both in expectation and with high probability, establishing sharp oracle inequalities, and extending aggregation theory to more general classes of weight constraints; see, e.g., \cite{Juditsky2000, Juditsky2008}, \cite{Yang2000, Yang2001}, \cite{Audibert2007}, \cite{Lounici2007}, \cite{Bunea2007}, \cite{Rigollet2007, Rigollet2011, Rigollet2012}, \cite{Dalalyan2007, DALALYAN20121423, Dalalyan2012Mirror}, \cite{Lecue2009erm, Lecue2013ew, Lecue2013convex,  Lecue2014Q-agg}, \cite{Dai2012Deviation}, \cite{Wang2014}, and \cite{Bellec2017, Bellec2019}. 

The aforementioned conventional statistical aggregation theory and its associated methodologies have provided important insights into high-dimensional learning, adaptive estimation, transfer learning, and many other areas of data science. While it accommodates general candidate procedures, a major limitation of this aggregation theory is that the candidate procedures to be aggregated are typically treated as deterministic, so that minimax optimality is assessed only at the aggregation stage. This framework is sometimes referred to as the \emph{pure aggregation} setting \citep[see, e.g.,][]{Rigollet2007, Rigollet2012KL}. In practical applications, implementing such aggregation procedures therefore often requires sample splitting: one subsample is used to construct the candidate estimators, while the other is reserved for aggregation. 

The pure aggregation framework does not cover an important and practically relevant setting in which the candidate estimators are themselves constructed from the same sample that is subsequently used for aggregation. A canonical example arises in model selection, where a collection of least squares estimators is fitted to the observed data and the same sample is subsequently used to select among them based on an information criterion \citep[see, e.g.,][]{Mallows1973, Akaike1974, Schwarz1978, Kneip1994Ordered, Foster1994, Barron1999MS, Yang1999sinica, baraud2000model, Baraud2002random, birge2001gaussian, Birge2007}. However, when the goal is to aggregate such data-dependent estimators rather than select a single one, the corresponding minimax theory and the construction of optimal aggregation strategies remain far less developed. To the best of our knowledge, existing work on this problem has focused almost exclusively on the MS aggregation regime, where the goal is to construct an aggregate that mimics the best single estimator among a given collection. This setting has been studied for least squares estimators \citep{Leung2006, Giraud2008, Alquier2011}, ordered linear smoothers \citep{Chernousova2013, Bellec2020}, and general affine estimators \citep{Dalalyan2012, Dai2014, Bellec2018}. In particular, \cite{Bellec2018} established the minimax optimality of Q-aggregation for affine estimators under mild conditions, showing that the optimal MS aggregation rate of affine estimators remains of the same order as in the deterministic candidate setting. 

For the aggregation of affine estimators beyond the MS aggregation regime, the existing literature remains relatively sparse, and the corresponding minimax rates have yet to be established. Under the convex aggregation regime, \cite{Peng2023shrinkage} and \cite{Peng2024Optimality} established nearly sharp oracle inequalities for aggregating least squares estimators from nested linear subspaces by minimizing a Mallows' $C_p$ criterion \citep{Mallows1973}. More recently, \cite{peng2025mallows} considered more general settings for convex and linear aggregations of least squares estimators, including estimators associated with general linear subspaces and all-subset collections. \cite{Bellec2018} also derived several sharp oracle inequalities for convex aggregation of general affine estimators. However, it remains unclear whether these upper bounds are sharp enough to attain the fundamental statistical limit due to the absence of the minimax lower bound. Moreover, how to construct optimal aggregation procedures for linear aggregation is also open. 

Motivated by the above discussion, this paper aims to establish the minimax rates for aggregating affine estimators beyond the MS aggregation regime. The importance of this problem is threefold. 
\begin{enumerate}
  \item Convex and linear aggregations constitute two fundamental aggregation regimes, and characterizing their minimax rates fills an important gap in the theory of aggregation for estimators from the same data. 
  \item Convex and linear aggregations target more ambitious benchmarks than MS aggregation and can lead to substantial performance improvements in some settings. As discussed in the Introduction of \cite{Peng2022improvability}, when substantial biases among the candidate estimators can be offset through convex combinations, relaxing the weight constraint from MS to convex aggregation can yield substantial improvements. Even when such a bias-cancellation advantage is absent, as in nested model settings, \cite{Peng2022improvability} showed that the minimum risk over convex combinations can still be substantially smaller than the minimum risk over individual estimators under certain conditions. 
  \item Focusing on affine estimators is of broad interest since this class encompasses many commonly used procedures in statistics and machine learning, including linear sieve estimators \citep{Newey1997}, Pinsker-type estimators \citep{Pinsker1980}, local polynomial estimators \citep{FanGijbels1996}, kernel ridge regression \citep{CaponnettoDeVito2007}, and random feature regression \citep{RahimiRecht2007}. The broad relevance of affine estimators is also supported by classical minimax theory. \cite{Donoho1990} showed that, over solid orthosymmetric quadratically convex parameter spaces, the minimax risk among linear estimators is within a factor of $1.25$ of the minimax risk among all (non-linear) estimators. 
\end{enumerate}


\subsection{Contributions}

In this paper, we study the convex and linear aggregation of general affine estimators constructed from the same data used to determine the aggregation weights. Our main contributions are summarized as follows. 

\begin{enumerate}
  \item In the convex-aggregation regime, we establish matching minimax lower and upper bounds in both probability and expectation. The lower bound is proved for zero-intercept affine estimators, whereas the upper bound requires only mild uniform bounds on the operator norms of the linear components and on the norms of the intercept vectors. The resulting minimax rate coincides with that for aggregating deterministic vectors, showing that the data dependence of affine estimators does not increase the fundamental difficulty of convex aggregation. 
  \item We prove that minimizing Mallows' $C_p$ criterion over the simplex yields a minimax optimal convex aggregate. Our proof exploits the strong convexity of the criterion in the fitted value of aggregate, together with a localized empirical process indexed by a seminorm. For the large-$M$ regime, we combine uniform concentration over line segments with two different applications of the Maurey sampling argument: one to approximate the oracle risk and another to sparsify the $C_p$ minimizer. These techniques substantially sharpen existing analyses of Mallows-type aggregation. 
  \item In the linear-aggregation regime, we establish matching minimax lower and upper bounds over a class of affine-estimator collections characterized by a lower bound on their minimum stable rank. We propose a truncated Mallows $C_p$ procedure and prove that it attains the optimal aggregation rate over this class.
\end{enumerate}

\subsection{Other related work}

The problem studied in this paper is closely related to the extensive econometrics literature on Mallows model averaging, which combines a collection of least squares estimators by minimizing Mallows' $C_p$ criterion. The seminal work of \cite{Hansen2007least} considered least squares estimators associated with nested linear models, while \cite{WAN2010277} and \cite{Zhang_2021} extended this framework to more general collections of least squares estimators. Our work differs from this literature in several aspects. First, we consider general affine estimators, which include least squares estimators as a special case. Second, we develop a non-asymptotic minimax theory that provides both upper bounds and matching lower bounds. Third, our results provide a statistical aggregation justification of Mallows model averaging and establish its optimality in terms of minimax excess loss and risk.

In the convex aggregation regime, the insightful works of \cite{WAN2010277} and \cite{Zhang_2021} established the asymptotic optimality of Mallows model averaging, showing that its loss asymptotically matches that of the best convex combination under suitable restrictions on the number and structure of the candidate models. Such asymptotic results have already provided important insights into the properties of Mallows model averaging. However, they do not identify the minimax excess-loss rate or determine whether the imposed restrictions are intrinsic to the problem. Our results strengthen this literature by characterizing the minimax rates of convex aggregation and showing that minimizing Mallows' $C_p$ criterion attains these rates under mild conditions. 

\subsection{Organization}

The remainder of this paper is organized as follows. Section~\ref{sec:setup} introduces the regression framework for aggregating affine estimators, formulates the notion of minimax optimality, and discusses the main gaps in the existing literature. Section~\ref{sec:convex} presents the minimax rates and optimal procedures for convex aggregation. The corresponding results for linear aggregation are developed in Section~\ref{sec:linear_agg}. Section~\ref{sec:disc} discusses the implications and limitations of our results. All proofs are deferred to the Appendix.

\section{Problem setup}\label{sec:setup}

\subsection{Setup and notation}\label{subsec:setup}

We observe a response vector $\bY \triangleq (Y_1, \ldots, Y_n)^{\top} \in \mathbb{R}^n$ generated from the fixed-design regression model 
\begin{equation}\label{eq:model}
  Y_i = f_i + \xi_i, \quad i = 1, \ldots, n, 
\end{equation}
where $\bxi \triangleq (\xi_1, \ldots, \xi_n)^{\top}$ consists of $n$ i.i.d. mean-zero Gaussian or sub-Gaussian random errors with $\mathbb E[\xi_i^2]=\sigma^2$, $f_i \triangleq f(x_i)$, $f: \mathcal{X} \to \mathbb{R}$ is an unknown regression function, and $x_1, \ldots, x_n \in \mathcal{X}$ are deterministic design points. Our goal is to estimate the unknown regression mean vector $\bbf \triangleq (f_1, \ldots, f_n)^{\top}$, which is assumed to belong to a subspace $\mathcal{F} \subseteq \mathbb{R}^n$. For an estimator $\hat{\bbf} \triangleq (\hat f_1,\ldots,\hat f_n)^\top$, we measure its performance by the empirical quadratic loss $L_n(\hat{\bbf}, \bbf) \triangleq \| \hat{\bbf} - \bbf \|_n^2 \triangleq (1/n) \sum_{i=1}^{n}(\hat{f}_i - f_i)^2$ and the corresponding risk $R_n(\hat{\bbf}, \bbf) \triangleq \mathbb{E}_{\bbf}\| \hat{\bbf} - \bbf \|_n^2$. This fixed-design regression framework is standard in the aggregation literature \citep[see, e.g.,][]{Leung2006, Rigollet2011, Dalalyan2012, Bellec2018, Peng2024Optimality}. 

Let $M=M_n\geq 2$ be an integer that may depend on the sample size $n$. To estimate the unknown mean vector $\bbf$, we consider aggregating $M$ affine estimators of the form  
\begin{equation}\label{eq:affine_est}
  \hat{\bbf}_m=\bA_m \bY+\bbb_m, \quad m=1, \ldots, M, 
\end{equation}
where $\bA_1,\ldots,\bA_M\in\mathbb{R}^{n\times n}$ are deterministic matrices, and $\bbb_1,\ldots,\bbb_M\in\mathbb{R}^n$ are deterministic intercept vectors. Let $\mathcal{M} \triangleq \{ (\bA_1,\bbb_1),\ldots,(\bA_M,\bbb_M) \}$ denote a collection of $M$ candidate affine estimators. Let $\mathbb{M} \subseteq (\mathbb{R}^{n\times n}\times\mathbb{R}^n)^M$ denote a family of such candidate collections, which serves as the parameter space for $\mathcal{M}$. Each element of the family $\mathbb{M}$ corresponds to a particular collection $\mathcal{M}$ of $M$ affine estimators.

Let $\mathcal{W} \subseteq \mathbb{R}^M$ be the weight constraint. For any $\bw \triangleq (w_1, \ldots, w_M)^{\top} \in \mathcal{W}$, the aggregated estimator of the affine estimators in $\mathcal{M}$ is defined as 
\begin{equation}\label{eq:agg_est}
  \hat{\bbf}_{\bw|\mathcal{M}} = \sum_{m = 1}^{M} w_m \hat{\bbf}_m = \bA_{\bw} \bY+\bbb_{\bw}, 
\end{equation}
where $\bA_{\bw} \triangleq \sum_{m=1}^{M}w_m\bA_m$, $\bbb_{\bw} \triangleq \sum_{m=1}^{M}w_m\bbb_m$, and the subscript $\bw|\mathcal{M}$ emphasizes that the aggregate $\hat{\bbf}_{\bw|\mathcal{M}}$ depends on both the weight vector $\bw$ and the candidate collection $\mathcal{M}$. And the performance of $\hat{\bbf}_{\bw|\mathcal{M}}$ is evaluated by $L_n(\hat{\bbf}_{\bw|\mathcal{M}}, \bbf) \triangleq \| \hat{\bbf}_{\bw|\mathcal{M}} - \bbf \|_n^2$ and $R_n(\hat{\bbf}_{\bw|\mathcal{M}}, \bbf) \triangleq\mathbb{E}_{\bbf}\| \hat{\bbf}_{\bw|\mathcal{M}} - \bbf \|_n^2$, respectively, where the expectation $\mathbb{E}_{\bbf}$ is taken with respect to the data in (\ref{eq:model}). 

The three most important weight constraints are 
\begin{equation}\label{eq:weight_constraint}
  \mathcal{W}_{\mathrm{MS}} \triangleq \left\{\be_1,\ldots,\be_M\right\}, \quad 
  \mathcal{W}_{\mathrm{C}} \triangleq \left\{\bw \in [0,1]^M:\sum_{m=1}^M w_m=1\right\}, \quad
  \mathcal{W}_{\mathrm{L}} \triangleq \mathbb{R}^M, 
\end{equation}
where $\be_m$ denotes the $m$-th canonical basis vector in $\mathbb{R}^M$. These three weight sets correspond to the MS, convex, and linear aggregation regimes, respectively. For any $\mathcal{W} \in \{\mathcal{W}_{\mathrm{MS}}, \mathcal{W}_{\mathrm{C}}, \mathcal{W}_{\mathrm{L}} \}$, the optimal performance by aggregation over a given candidate collection $\mathcal{M}$ is characterized by the \emph{oracle loss} $\inf_{\bw \in \mathcal{W}} L_n(\hat{\bbf}_{\bw|\mathcal{M}}, \bbf)$ and the \emph{oracle risk} $\inf_{\bw \in \mathcal{W}} R_n(\hat{\bbf}_{\bw|\mathcal{M}}, \bbf)$. It is immediate from $\mathcal{W}_{\mathrm{MS}} \subseteq \mathcal{W}_{\mathrm{C}} \subseteq \mathcal{W}_{\mathrm{L}}$ that 
\begin{equation*}
  \inf_{\bw \in \mathcal{W}_\mathrm{MS}} R_n(\hat{\bbf}_{\bw|\mathcal{M}}, \bbf) \geq \inf_{\bw \in \mathcal{W}_\mathrm{C}} R_n(\hat{\bbf}_{\bw|\mathcal{M}}, \bbf) \geq \inf_{\bw \in \mathcal{W}_\mathrm{L}} R_n(\hat{\bbf}_{\bw|\mathcal{M}}, \bbf), 
\end{equation*}
where the inequalities can be strict in some important settings \citep{Peng2022improvability, Peng2023shrinkage, Peng2024Optimality}. The goal of aggregation is to construct an aggregate whose loss or risk is as close as possible to the corresponding oracle benchmark, uniformly over all $\bbf \in \mathcal{F}$ and $\mathcal{M} \in \mathbb{M}$.

\subsection{Definitions of minimax optimality}

To characterize the fundamental statistical limits of optimal aggregation over $\bbf \in \mathcal{F}$ and $\mathcal{M} \in \mathbb{M}$, we adopt the minimax framework introduced in \cite{Nemirovski2000} and \cite{Tsybakov2003}. A positive sequence $\psi_n(\mathcal{F}, \mathbb{M}, \mathcal{W})$ is called the minimax lower bound for aggregation under the weight constraint $\mathcal{W}$ if, for all sufficiently large $n$ and $M$, there exists a candidate collection $\mathcal{M} \in \mathbb{M}$ such that 
\begin{equation}\label{eq:loss_lower bound}
  \inf_{\tilde{\bbf}} \sup_{\bbf \in \mathcal{F}} \mathbb{P}_{\bbf} \left[  L_n(\tilde{\bbf}, \bbf) - \inf_{\bw \in \mathcal{W}} L_n(\hat{\bbf}_{\bw|\mathcal{M}}, \bbf) \geq c \psi_n(\mathcal{F}, \mathbb{M}, \mathcal{W})\right] \geq p, 
\end{equation}
where $c>0$ and $0<p<1$ are constants independent of $n$ and $M$, and the infimum is taken over all estimators $\tilde{\bbf}$ based on the observed data $\bY$. A related minimax lower bound in risk is defined as 
\begin{equation}\label{eq:risk_lower bound}
  \inf_{\tilde{\bbf}} \sup_{\bbf \in \mathcal{F}} \left[  R_n(\tilde{\bbf}, \bbf) - \inf_{\bw \in \mathcal{W}} R_n(\hat{\bbf}_{\bw|\mathcal{M}}, \bbf) \right] \geq c \bar{\psi}_n(\mathcal{F}, \mathbb{M}, \mathcal{W}),  
\end{equation}
for some constant $c>0$. 

Moreover, if the lower bounds in (\ref{eq:loss_lower bound})--(\ref{eq:risk_lower bound}) are attainable by an aggregate $\hat{\bbf}$ uniformly over all $\bbf\in\mathcal{F}$ and $\mathcal{M}\in\mathbb{M}$, then $\psi_n(\mathcal{F},\mathbb{M},\mathcal{W})$ and $\bar{\psi}_n(\mathcal{F},\mathbb{M},\mathcal{W})$ determine the minimax optimal aggregation rates in probability and in expectation, respectively. And the corresponding aggregation procedure is minimax optimal under the weight constraint $\mathcal{W}$. 

\subsubsection{MS aggregation regime}

In the MS aggregation regime where $\mathcal{W} = \mathcal{W}_{\mathrm{MS}}$, the minimax rate of aggregating general affine estimators is now well understood. Define 
\begin{equation}\label{eq:MS_rate}
  \psi_{n,M}^{\mathrm MS} \triangleq \log(M)/n. 
\end{equation}
Let $\mathbb{M}_{\mathrm{det},B} \triangleq \{ \{ (\boldsymbol{0},\bbb_1),\ldots,(\boldsymbol{0},\bbb_M) \}: \max_{1\leq j\leq M}\|\bbb_j\|_n\leq B \}$ denote the family of affine-estimator collections for which all linear components vanish, so that the candidate estimators reduce to deterministic vectors. In this setting, it is proved that $\psi_n(\mathcal{F}, \mathbb{M}_{\mathrm{det},B}, \mathcal{W}_{\mathrm{MS}}) \asymp \psi_{n,M}^{\mathrm MS}$ \citep{Tsybakov2003, Rigollet2012}, and a variety of aggregation procedures have been proved to attain this rate over $\mathcal{M} \in \mathbb{M}_{\mathrm{det},B}$. Let $\mathbb{M}_{\mathrm{proj,r1}}^{(0)} \triangleq \{ \{ (\bA_1,\boldsymbol{0}),\ldots,(\bA_M,\boldsymbol{0}) \}: \bA_m \text{ is a rank-one orthogonal projector} \}$ denote the family for which all intercept vectors are $\boldsymbol{0}$ and the linear components are rank-one orthogonal projection matrices. Proposition~2.1 of \cite{Bellec2018} shows that if $\mathcal{F}=\mathbb{R}^n$ and the family $\mathbb{M}$ contains either $\mathbb{M}_{\mathrm{det},B}$ or $\mathbb{M}_{\mathrm{proj,r1}}^{(0)}$, then MS aggregation of affine estimator has the minimax lower bound 
\begin{equation}\label{eq:affine_MS_rate}
  \psi_n(\mathcal{F}, \mathbb{M}, \mathcal{W}_{\mathrm{MS}}) \asymp \psi_{n,M}^{\mathrm MS}.
\end{equation}

Moreover, the rate in (\ref{eq:affine_MS_rate}) is indeed minimax optimal and can be attained over a fairly general family $\mathbb{M}$ of affine estimator collections. Specifically, Theorem~2.1 of \cite{Bellec2018} establishes that a Q-aggregation estimator $\hat{\bbf}_{\hat{\bw}_Q|\mathcal{M}}$ satisfies, with probability at least $1-\exp(-x)$, 
\begin{equation*}
  L_n(\hat{\bbf}_{\hat{\bw}_Q|\mathcal{M}}, \bbf) - \inf_{\bw \in \mathcal{W}_{\mathrm{MS}}} L_n(\hat{\bbf}_{\bw|\mathcal{M}}, \bbf) \leq \frac{C(x + \log M)}{n}, 
\end{equation*}
uniformly over all $\mathcal{M} \in \mathbb{M}_{\mathrm{op}}$, where $\mathbb{M}_{\mathrm{op}} \triangleq \{ \mathcal{M}: \max_{j \neq k}\| \bA_j - \bA_k \|_{\mathrm{op}} \leq \phi \}$, and $\phi>1$ is some constant. MS aggregation over more restrictive families of affine estimators has also been studied by \cite{Leung2006} and \cite{Dalalyan2012} using the exponential weighting strategies. These insightful results show that, under reasonably mild conditions on the linear components $\bA_1,\ldots,\bA_M$, MS aggregation of affine estimators has the same minimax rate as aggregation of deterministic candidate estimators. 

\subsubsection{Convex aggregation regime}\label{sec:convx_review}

In the convex aggregation regime where $\mathcal{W} = \mathcal{W}_{\mathrm{C}}$, the minimax rate of aggregation of general affine estimators is unknown. Specifically, define 
\begin{equation}\label{eq:convex_rate}
  \psi_{n,M}^{\mathrm C} \triangleq \begin{cases}
    M/n, & \mbox{if } M \leq \sqrt{n} \\
    \sqrt{\frac{\log( 1 + M/\sqrt{n} )}{n}}, & \mbox{if } M > \sqrt{n}.
  \end{cases}
\end{equation}
It is only known that the minimax-optimal rate for aggregating deterministic vectors exhibits an elbow phenomenon \citep{Tsybakov2003, Rigollet2012} as shown in (\ref{eq:convex_rate}): $
  \psi_n(\mathcal{F}, \mathbb{M}_{\mathrm{det},B}, \mathcal{W}_{\mathrm{C}}) \asymp \psi_{n,M}^{\mathrm C}. 
$
However, for convex aggregation of affine estimators with nonzero linear components, the corresponding minimax rate remains unknown. 

Existing work has focused solely on deriving upper bounds. Under the restriction that $\bbb_m = \boldsymbol 0$ and the matrices $\bA_m$ are orthogonal projectors associated with nested linear subspaces, \cite{Peng2023shrinkage} and \cite{Peng2024Optimality} showed that estimating the aggregation weights by minimizing a Mallows' $C_p$ criterion (see definition in (\ref{eq:Cp})) yields  
\begin{equation}\label{eq:MMA_1}
  \mathbb{E}L_n(\hat{\bbf}_{\hat{\bw}|\mathcal{M}}, \bbf)  \leq [1 + o(1)] \inf_{\bw \in \mathcal{W}_{\mathrm{C}}} R_n(\hat{\bbf}_{\bw|\mathcal{M}}, \bbf) + \frac{(\log n)^t}{n}, 
\end{equation}
for some $t > 0$. More recently, \cite{peng2025mallows} extended this result to general orthogonal projection matrices without the nested restriction and established the (nearly sharp) oracle inequality under the forth-moment condition on the random errors:  
\begin{equation}\label{eq:MMA_2}
  \mathbb{E}L_n(\hat{\bbf}_{\hat{\bw}|\mathcal{M}}, \bbf) \leq [1 + o(1)] \inf_{\bw \in \mathcal{W}_{\mathrm{C}}} R_n(\hat{\bbf}_{\bw|\mathcal{M}}, \bbf) + C\min\left\{ M^2/n, \sqrt{M/n} \right\}.  
\end{equation}
Similar Mallows aggregation and Q-aggregation estimators were also studied in Proposition~7.2 of \cite{Bellec2018} for affine estimators satisfying $\max_{1 \leq m \leq M}\| \bA_m \|_{\mathrm{op}} \leq 1$ and $\max_{1 \leq m \leq M}\| \bbb_m \|_{n} \leq B$. However, the resulting oracle inequalities have remainder terms that converge no faster than $n^{-1/2}$, regardless of the value of $M$.  
 
The discrepancy between the existing upper bounds for convex aggregation of affine estimators and the minimax rate in (\ref{eq:convex_rate}) known for deterministic vectors raises the following fundamental questions:  
\begin{description}
  \item[Q1.] Does incorporating general affine estimators fundamentally change the minimax rate of convex aggregation? In other words, is the minimax lower bound strictly larger when the candidate estimators have nontrivial linear components, i.e., $\bA_m \neq \boldsymbol{0}$? 
      
  \item[Q2.] If not, can one construct a convex aggregation procedure that attains the minimax rate in (\ref{eq:convex_rate}) under reasonable assumptions on the family of affine-estimator collections? 
       
  \item[Q3.] Is the Mallows aggregation estimator minimax optimal for convex aggregation? Are the suboptimal upper bounds in (\ref{eq:MMA_1})--(\ref{eq:MMA_2}) merely technical artifacts of the existing analysis, or do they reflect an inherent suboptimality of Mallows aggregation for affine estimators? 
\end{description}

Answers to Q1 and Q2 would fill an important gap in our understanding of the fundamental statistical limits of convex aggregation for affine estimators and develop new methodologies for pursuing convex aggregation benchmarks. Q3 focuses on a specific aggregation procedure and is also of substantial interest, given the widespread use of Mallows aggregation methods in the econometrics and forecasting literature \citep[see, e.g.,][]{Hansen2007least, WAN2010277, Zhang_2021}. 

\subsubsection{Linear aggregation regime}\label{subsec:linear_review}

In the linear aggregation regime where $\mathcal{W}=\mathcal{W}_{\mathrm{L}}$, define 
\begin{equation}\label{eq:linear_rate}
  \psi_{n,M}^{\mathrm L} \triangleq M/n. 
\end{equation}
Then, it is only known that for linear aggregation of deterministic vectors, the rate is given by $
  \psi_n(\mathcal{F}, \mathbb{M}_{\mathrm{det},B}, \mathcal{W}_{\mathrm{L}}) \asymp \psi_{n,M}^{\mathrm L} 
$, and this rate can be attained by the least squares estimator over a linear space of dimension $M$ \citep{Rigollet2011,Rigollet2012KL, Tsybakov2003}. In contrast, for linear aggregation of affine estimators, both the minimax lower bound and matching upper bounds remain unexplored. This open problem was also noted in Section~7.2 of \cite{Bellec2018}. 

\begin{description}
  \item[Q4.] What is the minimax rate for linear aggregation of affine estimators, and can this rate be attained under certain conditions on the matrices $\bA_1,\ldots,\bA_M$?  
\end{description}

Answers to Q1--Q4 extend the three classical regimes of statistical aggregation theory developed in \cite{Nemirovski2000} and \cite{Tsybakov2003} to the setting in which the candidate estimators and the aggregation procedure are constructed using the same dataset, with the present paper focusing on affine estimators. 

\section{Convex aggregation}\label{sec:convex}

\subsection{Lower bounds}\label{subsec:lower}

Although the minimax lower bound for deterministic-vector aggregation, $\psi_n(\mathcal{F},\mathbb{M}_{\mathrm{det},B},\mathcal{W}_{\mathrm{C}})$, automatically provides a lower bound for convex aggregation over any larger family $\mathbb{M}\supseteq\mathbb{M}_{\mathrm{det},B}$, it does not reveal whether the data-dependent linear components introduce additional statistical difficulty. We therefore also seek a minimax lower bound for the zero-intercept case, where the candidate estimators depend on $\bY$ only through their linear components, in order to capture the fundamental limits arising specifically from this data dependence. 

To formulate the result, we consider the parameter space $\mathcal F_n(F)
  \triangleq\{\bbf\in\R^n:\|\bbf\|_n\le F\}$, where $F>0$ is a fixed constant. We also define the family of affine-estimator collections with zero intercept vectors by 
\begin{equation}\label{eq:operator-class} 
  \mathbb M_{M,A}^{(0)} \triangleq \left\{ \{ (\bA_1,\boldsymbol{0}),\ldots,(\bA_M,\boldsymbol{0}) \}: \max_{1\le j\le M}\|\bA_j\|_{\op}\le A \right\}.
\end{equation}
where $A>1$ is a constant. The following theorem establishes the minimax lower bound for convex aggregation over $\bbf\in\mathcal{F}_n(F)$ and $\mathcal{M}\in\mathbb{M}_{M,A}^{(0)}$. Since this subsection concerns the negative side of the problem, we assume that $\bxi\sim N(\boldsymbol{0},\sigma^2\bI_n)$. Recall the definition of the rate $\psi_{n,M}^{\mathrm C}$ in (\ref{eq:convex_rate}). 

\begin{theorem}[Lower bound]\label{theo:lower-zero-intercept}
There exist positive constants
\(c_0,c_1,c_2\) and positive integers \(n_0,M_0\), depending only on
\(F\) and $A$, such that the following statements hold whenever 
$
  n\ge n_0$, $M\ge M_0$, and
  $ \log M\le c_0n$.
There exists a deterministic candidate collection
\(\mathcal M\in\mathbb M_{M,A}^{(0)}\) such that:
\begin{equation}\label{eq:probability-lower}
  \inf_{\tilde{\bbf}}
  \sup_{\bbf\in\mathcal F_n(F)}
  \mathbb P_{\bbf}\!\left[
    L_n(\tilde{\bbf},\bbf)
    -\inf_{\bw \in \mathcal{W}_{\mathrm{C}}} L_n(\hat{\bbf}_{\bw|\mathcal{M}}, \bbf)
    \ge c_1\sigma^2\psi_{n,M}^{\mathrm C}
  \right]
  \ge c_2.
\end{equation}
For the same candidate collection, 
\begin{equation}\label{eq:risk-lower}
  \inf_{\tilde{\bbf}}
  \sup_{\bbf\in\mathcal F_n(F)}
  \left\{
    R_n(\tilde{\bbf},\bbf)
    -\inf_{\bw \in \mathcal{W}_{\mathrm{C}}} R_n(\hat{\bbf}_{\bw|\mathcal{M}}, \bbf)
  \right\}
  \ge c_1\sigma^2\psi_{n,M}^{\mathrm C},
\end{equation}
after possibly decreasing the constant \(c_1\).
\end{theorem}

Combining Theorem~\ref{theo:lower-zero-intercept} with the discussion in Section~\ref{sec:convx_review}, we see that, for any family $\mathbb{M}$ of affine-estimator collections satisfying either $\mathbb{M}\supseteq\mathbb{M}_{\mathrm{det},B}$ or $\mathbb{M}\supseteq\mathbb{M}_{M,A}^{(0)}$, the rate $\psi_{n,M}^{\mathrm C}$ provides a minimax lower bound for convex aggregation. 

\subsection{Upper bounds}\label{subsec:upper}

In the following, we show that the minimax lower bounds established in Theorem~\ref{theo:lower-zero-intercept} are attained by a convex aggregation procedure based on minimizing the Mallows' $C_p$ criterion \citep{Mallows1973}. Specifically, the aggregation weight vector $\bw$ in $\hat{\bbf}_{\bw|\mathcal{M}}$ is estimated by minimizing  \begin{equation}\label{eq:Cp}
C_p(\bw|\mathcal{M}, \bY)
\triangleq
\|\hat{\bbf}_{\bw|\mathcal{M}}-\bY\|^2
+2\sigma^2\tr(\bA_{\bw}).
\end{equation}
Let $
\hat{\bw}
\in
\argmin_{\bw\in\mathcal{W}_{\mathrm{C}}}C_p(\bw|\mathcal{M}, \bY) 
$ denote any minimizer of (\ref{eq:Cp}) over $\mathcal{W}_{\mathrm{C}}$. The following theorems show that the resulting aggregate $\hat{\bbf}_{\hat{\bw}|\mathcal{M}}$ attains the minimax lower bounds established in Theorem~\ref{theo:lower-zero-intercept}. Since this subsection focuses on the positive side of the problem, we assume that $\xi_i$ are i.i.d. sub-Gaussian random variables satisfying $\|\xi_i\|_{\psi_2}\leq \kappa\sigma$ for constant \(\kappa\geq1\), where
$
\|\xi\|_{\psi_2}
\triangleq
\inf\{t>0:\mathbb E\exp(\xi^2/t^2)\leq2\}.
$

For any constant $A\geq 1$, define the family of affine-estimator collections whose linear components satisfy a uniform operator-norm bound by  
\begin{equation}\label{eq:operator-class-2}
  \mathbb M_{M,A} \triangleq \left\{ \{ (\bA_1,\bbb_1),\ldots,(\bA_M,\bbb_M) \}: \max_{1\le j\le M}\|\bA_j\|_{\op}\le A, \bbb_j \in \mathbb{R}^n \right\}.
\end{equation}
Note that the famility $\mathbb M_{M,A}$ imposes no restriction on the intercept vectors. In particular, it contains the zero-intercept family defined in (\ref{eq:operator-class}) and used in the lower bound construction.
 
\begin{theorem}[Upper bound]\label{theo:upper_convex_1}
Suppose the data are generated from the model (\ref{eq:model}). There is a constant \(C_{\kappa,A}>0\), depending only on \(\kappa\) and \(A\), such that, uniformly for all $\bbf \in \mathbb{R}^n$ and $\mathcal{M} \in \mathbb M_{M,A}$, for every \(x\geq0\),
\begin{equation}\label{eq:upper_small_prob}
\mathbb P_{\bbf}\left[
L_n(\hat{\bbf}_{\hat{\bw}|\mathcal{M}},\bbf)
-\inf_{\bw \in \mathcal{W}_{\mathrm{C}}} L_n(\hat{\bbf}_{\bw|\mathcal{M}}, \bbf)
>
C_{\kappa,A}\frac{\sigma^2(M+x)}{n}
\right]
\leq \exp(-x).
\end{equation}
Moreover, the upper bound in risk is also available: 
\begin{align}
\mathbb{E}L_n(\hat{\bbf}_{\hat{\bw}|\mathcal{M}},\bbf)
&\leq
\mathbb E\inf_{\bw\in\mathcal{W}_{\mathrm{C}}}
      L_n(\hat{\bbf}_{\bw|\mathcal{M}}, \bbf)
+C_{\kappa,A}\frac{\sigma^2M}{n}
\label{eq:upper_small_random}\\
&\leq
\inf_{\bw \in \mathcal{W}_{\mathrm{C}}} R_n(\hat{\bbf}_{\bw|\mathcal{M}}, \bbf)
+C_{\kappa,A}\frac{\sigma^2M}{n}.
\label{eq:upper_small_risk}
\end{align}
\end{theorem}

We first provide a brief remark on the proof of Theorem~\ref{theo:upper_convex_1}. The proof of this theorem relies on three key ideas. First, in Lemma~\ref{lem:strong_conv} of the Appendix, we explore the strong convexity of the Mallows' $C_p$ criterion, which produces a negative quadratic term that is essential for deriving a sharp oracle inequality. Second, in Step~2 of the proof of Theorem~\ref{theo:upper_convex_1}, the stochastic remainder is localized through a seminorm, which captures both signal and noise fluctuations. Third, mixed-tail chaining and peeling balance the localized fluctuations against this negative quadratic drift. This localized approach sharpens the earlier global analyses in \cite{Bellec2018} about the Mallows aggregation estimator and related works of \cite{peng2025mallows}, while avoiding restrictions on the mean vector and intercept vectors. 

We then discuss the implications of Theorem~\ref{theo:upper_convex_1} for Q1--Q2 raised in Section~\ref{sec:convx_review}. Note that the leading constants in front of the oracle loss and oracle risk in Theorem~\ref{theo:upper_convex_1} are all equal to one. Therefore, the upper bounds in (\ref{eq:upper_small_prob}), (\ref{eq:upper_small_random}), and (\ref{eq:upper_small_risk}) yield sharp oracle inequalities, in high probability and in expectation, for convex aggregation based on minimizing (\ref{eq:Cp}). Moreover, the remainder terms in Theorem~\ref{theo:upper_convex_1} are of order $M/n$, which matches the minimax rate $\psi_{n,M}^{\mathrm C}$ in (\ref{eq:convex_rate}) when \(M\leq\sqrt n\). Thus, in the regime $M\leq\sqrt n$, Theorems~\ref{theo:lower-zero-intercept} and~\ref{theo:upper_convex_1} together establish that the minimax rates for convex aggregation in probability and in expectation are of the same order: 
\begin{equation*}
  \psi_n(\mathbb{R}^n,  \mathbb M_{M,A}, \mathcal{W}_{\mathrm{C}}) \asymp \bar{\psi}_n(\mathbb{R}^n,  \mathbb M_{M,A}, \mathcal{W}_{\mathrm{C}}) \asymp \psi_{n,M}^{\mathrm C}. 
\end{equation*}

For the large-\(M\) regime, we impose slightly stronger conditions on $\bbf$ and on the family of candidate affine-estimator collections. Define 
\begin{equation}\label{eq:operator-class-3}
  \mathbb M_{M,A,B} \triangleq \left\{ \{ (\bA_1,\bbb_1),\ldots,(\bA_M,\bbb_M) \}: \max_{1\le j\le M}\|\bA_j\|_{\op}\le A, \max_{1\leq j\leq M}\|\bbb_j\|_n\leq B \right\}, 
\end{equation}
where $A \geq 1$ and $B>0$ are constants. Compared with (\ref{eq:operator-class-2}), the family $\mathbb M_{M,A,B}$ additionally imposes a uniform bound on the intercept vectors. A similar restriction to (\ref{eq:operator-class-3}) was also imposed in Section~7.2 of \cite{Bellec2018}, whereas the resulting upper bound for the estimator $\hat{\bbf}_{\hat{\bw}|\mathcal M}$ is suboptimal. The following theorem shows that convex aggregation of affine estimators over $\mathbb M_{M,A,B}$ by minimizing Mallows' $C_p$ criterion is indeed minimax optimal.

\begin{theorem}[Upper bound]\label{theo:upper_convex_2}
  Suppose the data are generated from model~(\ref{eq:model}) with the sub-Gaussian errors. Let $\Lambda^2 \triangleq \sigma^2+F^2+B^2$. There exists a constant $C_{\kappa, A}>0$, depending only on $\kappa$ and $A$, such that, uniformly over all $\bbf \in \mathcal{F}_n(F)$ and $\mathcal{M} \in$ $\mathbb{M}_{M, A, B}$, for every integer $1 \leq m \leq M$ and every $x \geq 0$,
  \begin{equation}\label{eq:maurey_general_tail}
\mathbb P_{\bbf}\left[
L_n(\hat{\bbf}_{\hat{\bw}|\mathcal M},\bbf)
-\inf_{\bw\in\mathcal W_{\mathrm C}}
L_n(\hat{\bbf}_{\bw|\mathcal M},\bbf) >
C_{\kappa,A}
\left[
\Lambda^2
\left\{
\frac1m+
\frac mn\log\left(\frac{2eM}{m}\right)
\right\}
+
\frac{\sigma^2x}{n}
\right]
\right]
\leq
\exp(-x).
\end{equation}
Consequently, if $M>\sqrt{n}$ and $\log (e M / \sqrt{n}) \leq n$, then, for every $x \geq 0$, 
\begin{equation}\label{eq:maurey_rate_probability}
\mathbb P_{\bbf}\left[
L_n(\hat{\bbf}_{\hat{\bw}|\mathcal M},\bbf)
-
\inf_{\bw\in\mathcal W_{\mathrm C}}
L_n(\hat{\bbf}_{\bw|\mathcal M},\bbf)
>
C_{\kappa,A}
\left\{
\Lambda^2
\sqrt{\frac{\log(eM/\sqrt n)}{n}}
+
\frac{\sigma^2x}{n}
\right\}
\right]
\leq
\exp(-x),
\end{equation}
and
\begin{equation}\label{eq:maurey_rate}
\mathbb E L_n(\hat{\bbf}_{\hat{\bw}|\mathcal M},\bbf)
\leq
\inf_{\bw\in\mathcal W_{\mathrm C}}
R_n(\hat{\bbf}_{\bw|\mathcal M},\bbf)
+
C_{\kappa,A}\Lambda^2
\sqrt{\frac{\log(eM/\sqrt n)}{n}}.
\end{equation}
\end{theorem}

The proof combines a uniform concentration inequality over line segments with two applications of Maurey sampling: the first approximates the realized convex oracle, while the second sparsifies the estimated Mallows weight vector. Although Maurey sampling has previously been used for convex aggregation of deterministic candidates \citep[see, e.g.,][]{yang2004aggregating, Lecue2013convex, Bellec2019} and affine estimators \citep{Bellec2018}, the key novelty here is its twofold use together with uniform control over all grid-to-grid segments. This combination yields a sharp oracle inequality in Theorem~\ref{theo:upper_convex_2} when $M$ is large.  

Combining the minimax lower bound in Theorem~\ref{theo:lower-zero-intercept} with the upper bounds in Theorems~\ref{theo:upper_convex_1}--\ref{theo:upper_convex_2} resolves Questions Q1--Q3. Regarding Q1, allowing the candidates to be affine estimators constructed from the same data does not alter the minimax rate of convex aggregation in terms of \(n\) and \(M\). More precisely, for fixed \(F\), \(A\), \(B\), and \(\sigma>0\),
\begin{equation*}
\psi_n\bigl(\mathcal{F}_n(F),\mathbb M_{M,A,B},\mathcal{W}_{\mathrm C}\bigr)
\asymp
\bar{\psi}_n\bigl(\mathcal{F}_n(F),\mathbb M_{M,A,B},\mathcal{W}_{\mathrm C}\bigr)
\asymp
\sigma^2\psi_{n,M}^{\mathrm C}.
\end{equation*}
Regarding Q2, minimizing Mallows' \(C_p\) criterion over the simplex yields a minimax-optimal aggregation procedure, both in probability and in expectation, uniformly over \(\bbf\in\mathcal{F}_n(F)\) and \(\mathcal M\in\mathbb M_{M,A,B}\). Finally, regarding Q3, the suboptimal remainders in (\ref{eq:MMA_1})--(\ref{eq:MMA_2}) are technical rather than intrinsic. They can be sharpened by exploiting the localized empirical-process structure under sub-Gaussian assumption and, in the large-\(M\) regime, a refined argument that combines uniform line-segment control with two applications of Maurey sampling. 

%
%
%
%

\section{Linear aggregation}\label{sec:linear_agg}

In this section, we study the linear aggregation of $M$ affine estimators
under the weight constraint $\mathcal{W}_{\mathrm{L}}=\R^M$. As discussed in
Section~\ref{subsec:linear_review}, the minimax rates and optimal procedures
for linear aggregation of affine estimators remain unresolved. We
address this problem by establishing matching lower and upper bounds under a
\emph{minimum-stable-rank} condition on the candidate collection.

\subsection{Lower bounds}

We first introduce some additional notations. For a candidate collection
$\mathcal M=\{(\bA_j,\bbb_j):1\leq j\leq M\}$, define the linear map $T_{\mathcal M}:\R^M\to\R^{n\times n}\times\R^n$ by $T_{\mathcal M}(\bw) \triangleq (\bA_{\bw},\bbb_{\bw})$. The linear space generated by $\mathcal M$ is $\mathcal V(\mathcal M) \triangleq \{ (\bA_{\bw},\bbb_{\bw}) : \bw\in\R^M \}$, and its effective dimension is given by 
$
  D(\mathcal M)
  \triangleq
  \dim\bigl(\mathcal V(\mathcal M)\bigr)
  =
  \operatorname{rank}(T_{\mathcal M})
  \leq M.
$
Let
$
  \mathcal A(\mathcal M)
  \triangleq
  \left\{
    \bA\in\R^{n\times n}:
    (\bA,\bbb)\in\mathcal V(\mathcal M)
    \text{ for some }\bbb\in\R^n
  \right\}.
$
For a nonzero matrix $\bA$, recall that its stable rank is defined as 
$
  \operatorname{sr}(\bA)
  \triangleq
  \|\bA\|_{\F}^2/\|\bA\|_{\op}^2.
$
We define the minimum stable rank over the space $\mathcal A(\mathcal M)$ by
\begin{equation}\label{eq:minimum-stable-rank}
  s_{\mathcal A}(\mathcal M)
  \triangleq
  \inf_{\substack{
    \bA\in\mathcal A(\mathcal M)\\
    \bA\neq\boldsymbol{0}
  }}
  \operatorname{sr}(\bA).
\end{equation}
If $\mathcal A(\mathcal M)=\{\boldsymbol{0}\}$, we set
$s_{\mathcal A}(\mathcal M)=+\infty$. 

For integers $1\leq D\leq M$ and constants $A\geq1$, $B>0$, and $s\geq1$,
define
\begin{equation}\label{eq:diffuse-collection-class}
  \mathbb M_{M,D,A,B}^{s}
  \triangleq
  \left\{
    \mathcal M\in\mathbb M_{M,A,B}:
    D(\mathcal M)=D,
    \ s_{\mathcal A}(\mathcal M)\geq s
  \right\}.
\end{equation}
The family introduced in (\ref{eq:diffuse-collection-class}) is a subfamily of $\mathbb M_{M,A,B}$ for which the linear space generated by each candidate collection has effective dimension $D$. The condition $s_{\mathcal A}(\mathcal M)\geq s$ requires every nonzero matrix in $\mathcal A(\mathcal M)$ to have stable rank at least $s$, thereby preventing its Frobenius norm from being excessively concentrated in a single dominant singular direction.

The class $\mathbb M_{M,D,A,B}^{s}$ contains several natural candidate collections. First, if $\bA_j=\boldsymbol{0}$ for every $j$ and
$\dim\{\operatorname{span}(\bbb_1,\ldots,\bbb_M)\}=D$, then
$D(\mathcal M)=D$
  and
  $s_{\mathcal A}(\mathcal M)=+\infty$.
Thus the classical linear aggregation of deterministic vectors is included as a special case. Second, suppose
$
  \R^n=\mathcal H_1\oplus\cdots\oplus\mathcal H_D$ 
  and 
  $\dim(\mathcal H_\ell)\geq s$,
and let $\bP_\ell$ be the orthogonal projection onto $\mathcal H_\ell$.  
Define the nested projection estimators by 
$
  \bA_j=\sum_{\ell=1}^j\bP_\ell$,
  $
  \bbb_j=\boldsymbol{0}$,
  for 
  $1\leq j\leq D$. Then every matrix in $\mathcal A(\mathcal M)$ has the form
$
  \bA=\sum_{\ell=1}^D c_\ell\bP_\ell.
$
Consequently,
\[
  \frac{\|\bA\|_{\F}^2}{\|\bA\|_{\op}^2}
  =
  \frac{\sum_{\ell=1}^D
    \dim(\mathcal H_\ell)c_\ell^2}
       {\max_{\ell}c_\ell^2}
  \geq
  \min_{\ell}\dim(\mathcal H_\ell)
  \geq s.
\]

\begin{theorem}[Lower bound]
\label{theo:diffuse-linear-lower}
Suppose that $\bxi\sim N(\boldsymbol{0},\sigma^2\bI_n)$. Fix $F>0$,
$A\geq1$, and $B>0$. There exist universal constants $c_1,c_2>0$ such
that, for all integers $2\leq D\leq M$ with $D\leq n$ and every $s\geq1$,
there is a deterministic collection
$\mathcal M_0\in\mathbb M_{M,D,A,B}^{s}$ satisfying
\begin{align}
&\inf_{\tilde{\bbf}}
\sup_{\bbf\in\mathcal F_n(F)}
\mathbb P_{\bbf}\!\left[
  L_n(\tilde{\bbf},\bbf)
  -\inf_{\bw\in\R^M}
    L_n(\hat{\bbf}_{\bw|\mathcal M_0},\bbf)
  \geq
  c_1
  \left\{
    F^2\wedge\frac{\sigma^2D}{n}
  \right\}
\right]
\geq c_2,
\label{eq:diffuse-lower-prob}\\
&\inf_{\tilde{\bbf}}
\sup_{\bbf\in\mathcal F_n(F)}
\left\{
  R_n(\tilde{\bbf},\bbf)
  -\inf_{\bw\in\R^M}
    R_n(\hat{\bbf}_{\bw|\mathcal M_0},\bbf)
\right\}
\geq
c_1
\left\{
  F^2\wedge\frac{\sigma^2D}{n}
\right\}.
\label{eq:diffuse-lower-risk}
\end{align}
\end{theorem}

If $s\leq n$, the collection $\mathcal M_0$ can be chosen so that every candidate matrix is
the same nonzero full-rank matrix and
$s_{\mathcal A}(\mathcal M_0)=n$. For $s>n$, the same lower bounds hold using a
deterministic-vector subfamily, for which
$s_{\mathcal A}(\mathcal M_0)=+\infty$.

\subsection{Upper bounds}

To construct a linear aggregation estimator over the family $\mathbb M_{M,D,A,B}^{s}$, we consider a truncated Mallows aggregation procedure. Recall the Mallows' $C_p$ criterion defined in (\ref{eq:Cp}), and let $\hat{\bw}^{\mathrm{L}} \in \argmin_{\bw \in \mathbb{R}^M} C_p(\bw|\mathcal M,\bY)$ denote any weight vector that minimizes this criterion over $\R^M$. The corresponding linear aggregate is $\hat{\bbf}_{\mathcal M}^{\mathrm{Cp}}
  \triangleq \hat{\bbf}_{\hat{\bw}^{\mathrm{L}}|\mathcal M}$. If the minimum of the Mallows' $C_p$ criterion is
not attained, set
$\hat{\bbf}_{\mathcal M}^{\mathrm{Cp}}
  =\boldsymbol{0}$. For $r>0$, let $\Pi_r$ denote Euclidean projection onto the closed ball
$\{\bv\in\R^n:\|\bv\|\leq r\}$. For the parameter space
$\mathcal F_n(F)$, define the truncated Mallows linear aggregate by
\begin{equation}\label{eq:truncated-cp}
  \hat{\bbf}_{\mathcal M}^{\mathrm{TCp}}
  \triangleq
  \Pi_{F\sqrt n}
  \left(\hat{\bbf}_{\mathcal M}^{\mathrm{Cp}}
  \right).
\end{equation}
We then establish the following sharp oracle inequality for $\hat{\bbf}_{\mathcal M}^{\mathrm{TCp}}$. Here, we use notation $a \wedge b \triangleq \min\{a, b \}$. 

\begin{theorem}[Upper bound]
\label{theo:stable-upper}
Suppose that the errors satisfy the sub-Gaussian condition. Fix $F>0$, $A\geq1$, and $B>0$. There are constants
$c_\kappa,C_\kappa>0$, depending only on $\kappa$, such that the following
holds. Let
$\mathcal M\in\mathbb M_{M,D,A,B}^{s}$ and
$\bbf\in\mathcal F_n(F)$. Then, for every $x\geq0$,
\begin{align}
&\mathbb P_{\bbf}\!\left[
  L_n(\hat{\bbf}_{\mathcal M}^{\mathrm{TCp}},\bbf)
  -\inf_{\bw\in\R^M}
    L_n(\hat{\bbf}_{\bw|\mathcal M},\bbf)
  >
  C_\kappa
  \left[
    F^2\wedge
    \frac{\sigma^2}{n}
    \left\{
      D+x+\frac{(D+x)^2}{s}
    \right\}
  \right]
\right]
\nonumber\\
&\hspace{35mm}
\leq
2\exp(-x)
+2\exp\{C_\kappa D-c_\kappa s\}, 
\label{eq:stable-upper-prob}
\end{align}
where we define $(D+x)^2/(+\infty)=0$ and
$\exp(-c_\kappa\cdot+\infty)=0$. In particular, if
$
  s\geq C_\kappa(D+x),
$
then
\begin{align}
\mathbb P_{\bbf}\!\left[
  L_n(\hat{\bbf}_{\mathcal M}^{\mathrm{TCp}},\bbf)
  -\inf_{\bw\in\R^M}
    L_n(\hat{\bbf}_{\bw|\mathcal M},\bbf)
  >
  C_\kappa
  \left\{
    F^2\wedge\frac{\sigma^2(D+x)}{n}
  \right\}
\right]
&\leq4\exp(-x).
\label{eq:stable-upper-prob-simplified}
\end{align}
If
$
  s
  \geq
  C_\kappa
  \{
    D+
    \log(
      e+\frac{nF^2}{\sigma^2D}
    )
  \},
$
then
\begin{align}
R_n(\hat{\bbf}_{\mathcal M}^{\mathrm{TCp}},\bbf)
&\leq
\inf_{\bw\in\R^M}
  R_n(\hat{\bbf}_{\bw|\mathcal M},\bbf)
+C_\kappa
\left\{
  F^2\wedge\frac{\sigma^2D}{n}
\right\}.
\label{eq:stable-upper-risk}
\end{align}
\end{theorem}

Combining Theorems~\ref{theo:diffuse-linear-lower} and
\ref{theo:stable-upper} shows that, under $s \geq D$ (up to a logarithmic term in $n$), the minimax excess-risk rate over
$\mathbb M_{M,D,A,B}^{s}$ is
$$
  F^2\wedge\frac{\sigma^2D}{n}.
$$
When $D=M\leq n$ and $F$ and $\sigma$ are fixed, this reduces to the
classical linear-aggregation rate $\sigma^2M/n$ discussed in Section~\ref{subsec:linear_review}. Therefore, this section establishes that, over the family $\mathbb M_{M,D,A,B}^{s}$ of affine-estimator collections with $D=M\leq n$, the minimax rate for linear aggregation matches the optimal rate $\psi_{n,M}^{\mathrm L}$ for the deterministic-candidate setting.

\section{Discussion}\label{sec:disc}

Aggregating multiple estimation procedures to form a final decision is a fundamental strategy in econometrics, statistics, and machine learning. Earlier studies on the aggregation of affine estimators, including \cite{Dalalyan2012,Dai2012Deviation,Bellec2018}, have primarily focused on model-selection aggregation, whose objective is to perform nearly as well as the best individual affine estimator. By contrast, the corresponding problems of convex and linear aggregation have remained largely unresolved. This paper fills this theoretical gap by establishing the minimax rates and constructing minimax-optimal procedures for convex aggregation. We further develop matching lower and upper bounds for linear aggregation under suitable structural conditions on the candidate collection.

Our results also provide novel non-asymptotic statistical foundation for Mallows model averaging, which has been widely studied and applied in econometrics \citep{Hansen2007least,WAN2010277,Zhang_2021}. In particular, we show that minimizing Mallows' $C_p$ criterion is not merely asymptotically optimal relative to an infeasible optimal loss, but can attain the minimax aggregation rate in both probability and expectation. These findings support the use of Mallows model averaging in practical problems. 

The present theory relies essentially on the affine structure of the candidate estimators. For general nonlinear estimators constructed from the same data used for aggregation, the dependence between the candidates and the aggregation criterion may generate substantially more complicated stochastic terms. Without additional restrictions on the stability or complexity of the candidate procedures, a universal minimax rate depending only on $n$ and $M$ should not generally be expected. When such nonlinear dependence cannot be controlled directly, sample splitting remains a broadly applicable alternative, although it may reduce statistical efficiency by decreasing the effective sample size available for estimation and aggregation. 

An important direction for future research is to develop a more general theory of linear aggregation and under other weight constraints. The minimum-stable-rank condition used in this paper provides a sufficient mechanism for controlling unrestricted linear combinations, but it is not known whether this condition is necessary or can be substantially weakened. Determining the minimax rate over broader families of affine estimators, and constructing linear aggregation procedures that attain it under minimal structural assumptions, remain open problems.



\appendix

\section*{Appendix}
\addcontentsline{toc}{section}{Appendix}
\renewcommand{\thesection}{\Alph{section}}
\renewcommand{\thesubsection}{\thesection.\arabic{subsection}}
\renewcommand{\thesubsubsection}{\thesubsection.\arabic{subsubsection}}
\renewcommand{\theparagraph}{\thesubsubsection.\arabic{paragraph}}
\numberwithin{equation}{section}
\renewcommand{\theequation}{\thesection.\arabic{equation}}

\section{Proof of Theorem~\ref{theo:lower-zero-intercept}}

The following lemma provides the geometric ingredients required for the proof of Theorem~\ref{theo:lower-zero-intercept}. Its construction follows the two-regime argument used to establish the convex aggregation lower bound in \citet{Tsybakov2003}. 

\begin{lemma}[Convex packing in a Euclidean space]
\label{lem:convex-packing}
Let $\mathcal H$ be a Euclidean space of dimension $n-1$, let $0 < \kappa \le 1$ and $0 < \lambda \le 1/2$, and set
$
  \tau=\kappa\sigma\sqrt{\frac{3n}{4}},
  a=\lambda\tau.
$
Under the range of $(n,M)$ in Theorem~\ref{theo:lower-zero-intercept}, there exist vectors $\bg_1,\ldots,\bg_M\in\mathcal H$, with $\max_j\|\bg_j\|\le a$, and a finite set
$
  \Theta\subseteq\operatorname{conv}\{\bg_1,\ldots,\bg_M\}
$
such that, for every estimator $T$ based on
$$
  \bZ=\boldsymbol\theta+\boldsymbol\varepsilon,
  \qquad
  \boldsymbol\varepsilon\sim N(\boldsymbol{0},\sigma^2\bI_{\mathcal H}),
$$
where $\bI_{\mathcal{H}}$ denotes the identity operator on the $(n-1)$-dimensional subspace $\mathcal{H}$, 
we have
\begin{equation}\label{eq:packing-testing}
  \frac1{|\Theta|}
  \sum_{\boldsymbol\theta\in\Theta}
  \mathbb P_{\boldsymbol\theta}\!\left[
    \frac1n\|T-\boldsymbol\theta\|^2
    \ge c_*\kappa^2\lambda^2\sigma^2\psi_{n,M}^{\mathrm C}
  \right]
  \ge p_*,
\end{equation}
where $c_*,p_*>0$ are universal constants.
\end{lemma}

\begin{proof}[Proof of Lemma~\ref{lem:convex-packing}]

We give the construction in the two regimes.

\noindent\textsc{\underline{Case 1: \(M\le\sqrt n\).}} Let $d=\lfloor\frac{M-1}{2}\rfloor$ and choose orthonormal vectors $\bv_1,\ldots,\bv_d\in\mathcal H$. Consider the dictionary containing $M$ vectors: \begin{equation}\label{eq:dic_1}\{ \boldsymbol{0}, \ldots, \boldsymbol{0},
  a\bv_1,-a\bv_1,\ldots,a\bv_d,-a\bv_d\},\end{equation} where $a=\lambda\tau$, $\tau=\kappa\sigma\sqrt{\frac{3n}{4}}$, $0 < \kappa \le 1$, $0 < \lambda \le 1/2$, and repeated zero vectors $\boldsymbol{0}, \ldots, \boldsymbol{0}$ are added into the dictionary whenever $2d+1 < M$. Let $\delta=\gamma\kappa\lambda\sigma$, where $\gamma>0$ is a sufficiently
small universal constant, and define
\[
  \boldsymbol\theta_{\boldsymbol\omega}
  =\delta\sum_{j=1}^d\omega_j\bv_j,
  \qquad
  \boldsymbol\omega\in\{-1,1\}^d.
\]
Because
$
  \frac{d\delta}{a}
  =\frac{d\gamma}{\sqrt{3n/4}}
  \le1
$
for sufficiently small \(\gamma\), every
\(\boldsymbol\theta_{\boldsymbol\omega}\) belongs to the convex hull of
the displayed dictionary in (\ref{eq:dic_1}). Specifically, we observe that \(\omega_j a\bv_j\) is an element of (\ref{eq:dic_1}) for any $\omega_j \in \{-1,1\}$. Since $\frac{d\delta}{a} \le1$, we can put weight \(\delta/a\) on
\(\omega_j a\bv_j\) for each \(j\), and put the remaining weight on \(\boldsymbol{0}\)'s. 

By the Varshamov--Gilbert lemma \citep[see, e.g., Lemma 2.9 of][]{Tsybakov2009}, there is a subset
\(\Omega\subseteq\{-1,1\}^d\) such that
$
  \log|\Omega|\ge c d$ and $ 
  d_{\mathrm H}(\boldsymbol\omega,\boldsymbol\omega')\ge c d
$
for distinct elements of \(\Omega\). Consequently,
\begin{equation}\label{eq:small-separation}
  c\delta^2d
  \le
  \|\boldsymbol\theta_{\boldsymbol\omega}
       -\boldsymbol\theta_{\boldsymbol\omega'}\|^2
  \le4\delta^2d.
\end{equation}
The Kullback-Leibler divergence between the corresponding Gaussian distributions is
$
  \operatorname{KL}
  (P_{\boldsymbol\theta_{\boldsymbol\omega}},
   P_{\boldsymbol\theta_{\boldsymbol\omega'}})
  =\frac{
    \|\boldsymbol\theta_{\boldsymbol\omega}
      -\boldsymbol\theta_{\boldsymbol\omega'}\|^2
  }{2\sigma^2}
  \le2\gamma^2d.
$
Choose \(\gamma\) small enough that $2\gamma^2d$ is at most a fixed small
multiple of \(\log|\Omega|\). Theorem 2.5 of \cite{Tsybakov2009} and
\eqref{eq:small-separation} then give
\[
  \inf_T\frac1{|\Omega|}
  \sum_{\boldsymbol\omega\in\Omega}
  \mathbb P_{\boldsymbol\theta_{\boldsymbol\omega}}
  \!\left[
    \frac1n\|T-\boldsymbol\theta_{\boldsymbol\omega}\|^2
    \ge c\kappa^2\lambda^2\sigma^2\frac{M}{n}
  \right]
  \ge p_*.
\]

\noindent\textsc{\underline{Case 2: \(M>\sqrt n\).}} Put \(p=M-1\) and
$
  \ell_M=\log\left(\frac{eM}{\sqrt n}\right).
$
For a sufficiently small universal constant \(c_k>0\), choose an integer
$
  k\asymp\sqrt{\frac{n}{\ell_M}}$ and 
  $1\le k\le p/4,
$
so that
\begin{equation}\label{eq:k-balance}
  k^2\log(ep/k)\asymp n.
\end{equation}
The restriction \(\log M\le c_0n\), with \(c_0\) sufficiently small,
ensures that \(k\) can be chosen so that
$
  n-1\ge Ck\log(ep/k).
$
By \citet{Baraniuk2008}
(see also the construction in
\citealp[Section~5.2]{Rigollet2011}),
there exist vectors
$\bv_1,\ldots,\bv_p\in\mathcal H$ satisfying
$\|\bv_j\|\leq1$ and
\begin{equation}\label{eq:rip}
  \frac14\|\bu\|^2
\leq
\left\|\sum_{j=1}^p u_j\bv_j\right\|^2
\leq
2\|\bu\|^2
\end{equation}
for every $2k$-sparse $\bu\in\R^p$, provided that
$n-1\geq Ck\log(ep/k)$. 

Take
$
  \bg_1=\boldsymbol{0}, 
  \bg_{j+1}=a\bv_j,
   1\leq j\leq p.
$
A constant-weight Varshamov--Gilbert argument gives a collection
\(\mathfrak S\) of \(k\)-element subsets of \(\{1,\ldots,p\}\) such that
\begin{equation}\label{eq:set-packing}
  |S\triangle S'|
  \geq \frac{k}{2}
  \quad\text{for all distinct }S,S'\in\mathfrak S,
  \qquad
  \log|\mathfrak S|
  \geq c k\log\left(\frac{ep}{k}\right).
\end{equation}
For a sufficiently small universal constant \(\eta\in(0,1)\), define
\begin{equation}\label{eq:large-theta}
  \boldsymbol\theta_S
  =\frac{\eta}{k}\sum_{j\in S}\bg_{j+1},
  \qquad S\in\mathfrak S.
\end{equation}
This is a convex combination: its selected atoms receive weight
\(\eta/k\), and \(\bg_1=\boldsymbol{0}\) receives weight \(1-\eta\).
Equations \eqref{eq:rip}--\eqref{eq:set-packing} imply that, for
distinct \(S,S'\),
\begin{equation}\label{eq:large-separation}
  \frac{\eta^2a^2}{8k}
  \le
  \|\boldsymbol\theta_S-\boldsymbol\theta_{S'}\|^2
  \le
  \frac{4\eta^2a^2}{k}.
\end{equation}
It follows that
$
  \operatorname{KL}(P_{\boldsymbol\theta_S},
                    P_{\boldsymbol\theta_{S'}})
  \le
  \frac{2\eta^2a^2}{\sigma^2k}
  \le C\eta^2\frac nk.
$
On the other hand, \eqref{eq:k-balance} and
\eqref{eq:set-packing} give
$
  \log|\mathfrak S|
  \ge c k\log(ep/k)
  \ge c'\frac nk.
$
Choosing \(\eta\) sufficiently small makes the Kullback--Leibler
divergence a small fixed multiple of \(\log|\mathfrak S|\).  Fano's
lemma, \eqref{eq:large-separation}, and
\(k^{-1}\asymp\sqrt{\ell_M/n}\) yield
\[
  \inf_T\frac1{|\mathfrak S|}
  \sum_{S\in\mathfrak S}
  \mathbb P_{\boldsymbol\theta_S}\!\left[
    \frac1n\|T-\boldsymbol\theta_S\|^2
    \ge
    c\kappa^2\lambda^2\sigma^2
    \sqrt{\frac{\ell_M}{n}}
  \right]
  \ge p_*.
\]
Combining the two cases proves the lemma.
\end{proof}

\begin{proof}[Proof of Theorem~\ref{theo:lower-zero-intercept}]

\noindent\textsc{\underline{Step 1: Construction of candidate matrices.}}
Fix an arbitrary deterministic unit vector
\(\be_0\in\R^n\), and let
$
  \mathcal H=\be_0^\perp.
$
Under the Gaussian model, one may equivalently take
\(\be_0=(1,0,\ldots,0)^\top\).  Define \(\kappa_F=\min\{1,F/\sigma\}\) and
\begin{equation}\label{eq:correct-lambda}
  \lambda_A
  =\frac12\min\{1,\sqrt{A^2-1}\},
  \qquad
  \tau=\kappa_F\sigma\sqrt{\frac{3n}{4}},
  \qquad
  a=\lambda_A\tau.
\end{equation}
Apply Lemma~\ref{lem:convex-packing} in \(\mathcal H\) with
\(\kappa=\kappa_F\) and \(\lambda=\lambda_A\), obtaining
\(\bg_1,\ldots,\bg_M\) and \(\Theta\).  For \(1\le j\le M\), set
\begin{equation}\label{eq:pilot-matrix}
  \bA_j
  =\left(\be_0+\frac{\bg_j}{\tau}\right)\be_0^\top,
  \qquad
  \bbb_j=\boldsymbol{0}.
\end{equation}
These are deterministic rank-one matrices.  Because
\(\bg_j\perp\be_0\),
\begin{align*}
  \|\bA_j\|_{\op}
  &=\left\|\be_0+\frac{\bg_j}{\tau}\right\|=\sqrt{1+\frac{\|\bg_j\|^2}{\tau^2}}
  \le\sqrt{1+\lambda_A^2}
  \le A.
\end{align*}
Hence the resulting collection belongs to
\(\mathbb M_{M,A}^{(0)}\).

\medskip
\noindent\textsc{\underline{Step 2: Hard regression means.}}
For each \(\boldsymbol\theta\in\Theta\), define
\begin{equation}\label{eq:hard-mean}
  \bbf_{\boldsymbol\theta}
  =\tau\be_0+\boldsymbol\theta.
\end{equation}
Since \(\boldsymbol\theta\in\operatorname{conv}
\{\bg_1,\ldots,\bg_M\}\), we have
\(\|\boldsymbol\theta\|\le a\).  Therefore,
\[
  \|\bbf_{\boldsymbol\theta}\|_n^2
  =\frac{\tau^2+\|\boldsymbol\theta\|^2}{n}
  \le\frac{\tau^2+a^2}{n}
  =\frac34\kappa_F^2(1+\lambda_A^2)\sigma^2
  \le F^2.
\]
Thus every hard mean belongs to \(\mathcal F_n(F)\). The observation $\bY = \bbf_{\boldsymbol\theta} + \bxi$ decomposes orthogonally as
\begin{equation}\label{eq:orthogonal-data}
  Y_0=\be_0^\top\bY=\tau+\xi_0,
  \qquad
  \bY_{\mathcal H}=\bP_{\mathcal H}\bY
  =\boldsymbol\theta+\bxi_{\mathcal H}.
\end{equation}
We refer to \(Y_0=\be_0^\top\bY\) as the \emph{pilot coordinate}. It provides the common random scaling factor used in constructing the affine estimators by (\ref{eq:pilot-matrix}), whereas \(\bY_{\mathcal H}\) contains the information needed to identify \(\boldsymbol\theta\). Specifically, we have 
$
\xi_0\sim N(0,\sigma^2),
\bxi_{\mathcal H}\sim
 N(\boldsymbol{0},\sigma^2\bI_{\mathcal H}),
$
and \(\xi_0\) and \(\bxi_{\mathcal H}\) are independent. Moreover, the distribution of \(Y_0\) is \(N(\tau,\sigma^2)\) for every \(\boldsymbol\theta\in\Theta\). Hence, \(Y_0\) provides no information about which element of \(\Theta\) generated the data.

\medskip
\noindent\textsc{\underline{Step 3: The convex oracle has only pilot noise.}}
Fix \(\boldsymbol\theta\in\Theta\).  Select a deterministic
\(\bw_{\boldsymbol\theta}\in \mathcal{W}_{\mathrm{C}}\) such that
$
  \boldsymbol\theta
  =\sum_{j=1}^M w_{\boldsymbol\theta,j}\bg_j.
$
Using \eqref{eq:pilot-matrix} and \eqref{eq:hard-mean},
$
  \bA_{\bw_{\boldsymbol\theta}}
  =\frac{\bbf_{\boldsymbol\theta}}{\tau}\be_0^\top,\,
  \bA_{\bw_{\boldsymbol\theta}}\bY
  =\frac{Y_0}{\tau}\bbf_{\boldsymbol\theta}.
$
Consequently,
\begin{equation}\label{eq:oracle-realized}
  \inf_{\bw\in\mathcal{W}_{\mathrm{C}}}
  L_n(\hat{\bbf}_{\bw|\mathcal M},\bbf_{\boldsymbol\theta})
  \le
  \frac{\xi_0^2}{\tau^2}
  \|\bbf_{\boldsymbol\theta}\|_n^2
  \le
  (1+\lambda_A^2)\frac{\xi_0^2}{n}.
\end{equation}
Taking expectations also gives
\begin{equation}\label{eq:oracle-risk}
  \inf_{\bw\in\mathcal{W}_{\mathrm{C}}}
  R_n(\hat{\bbf}_{\bw|\mathcal M},\bbf_{\boldsymbol\theta})
  \le
  (1+\lambda_A^2)\frac{\sigma^2}{n}.
\end{equation}

\medskip
\noindent\textsc{\underline{Step 4: Probability lower bound.}}
Let \(\tilde{\bbf}=\tilde{\bbf}(\bY)\) be an arbitrary estimator
and put
$
  \tilde{\boldsymbol\theta}
  =\bP_{\mathcal H}\tilde{\bbf}.
$
For every fixed value \(Y_0=y_0\), the map
\(\bY_{\mathcal H}\mapsto
\tilde{\boldsymbol\theta}(y_0,\bY_{\mathcal H})\) is an estimator in
the Gaussian experiment of Lemma~\ref{lem:convex-packing}.  Since the
conditional distribution of \(\bY_{\mathcal H}\) does not depend on
\(y_0\), \eqref{eq:packing-testing} gives, uniformly in \(y_0\),
\begin{equation}\label{eq:conditional-fano}
  \frac1{|\Theta|}
  \sum_{\boldsymbol\theta\in\Theta}
  \mathbb P_{\boldsymbol\theta}\!\left[
    \frac1n
    \|\tilde{\boldsymbol\theta}
      -\boldsymbol\theta\|^2
    \ge r_{n,M}
    \ \middle|\ Y_0=y_0
  \right]
  \ge p_*,
\end{equation}
where
$
  r_{n,M}=c_*\kappa_F^2\lambda_A^2\sigma^2\psi_{n,M}^{\mathrm C}.
$

Choose a constant \(q_{F,A}>0\), depending only on \(F/\sigma\) and \(A\), so small that
\begin{equation}\label{eq:q-choice}
  (1+\lambda_A^2)q_{F,A}^2\frac{\sigma^2}{n}
  \le\frac12r_{n,M}
\end{equation}
for all admissible \((n,M)\).  This is possible because
\(n\psi_{n,M}^{\mathrm C}\ge1\).  Let
$
  \mathcal E_0=\{|\xi_0|\le q_{F,A}\sigma\}.
$
Its probability
\(q_0(F,A)=\mathbb P(|N(0,1)|\le q_{F,A})\) is strictly positive and does not
depend on \(n\), \(M\), or \(\boldsymbol\theta\). Integrating \eqref{eq:conditional-fano} only over
\(\mathcal E_0\) yields
\begin{equation}\label{eq:fano-and-pilot}
  \frac1{|\Theta|}
  \sum_{\boldsymbol\theta\in\Theta}
  \mathbb P_{\boldsymbol\theta}\!\left[
    \frac1n
    \|\tilde{\boldsymbol\theta}
      -\boldsymbol\theta\|^2
    \ge r_{n,M},\ \mathcal E_0
  \right]
  \ge p_*q_0(F,A).
\end{equation}
On the event inside \eqref{eq:fano-and-pilot}, orthogonal projection and
\eqref{eq:oracle-realized}--\eqref{eq:q-choice} imply
\begin{align*}
  &L_n(\tilde{\bbf},\bbf_{\boldsymbol\theta})
   -\inf_{\bw\in \mathcal{W}_{\mathrm{C}}}
      L_n(\hat{\bbf}_{\bw|\mathcal M},\bbf_{\boldsymbol\theta})\ge
  \frac1n
  \|\tilde{\boldsymbol\theta}
     -\boldsymbol\theta\|^2
  -(1+\lambda_A^2)\frac{\xi_0^2}{n}
  \ge\frac12r_{n,M}.
\end{align*}
Therefore, for every \(\tilde{\bbf}\),
\[
  \sup_{\bbf\in\mathcal F_n(F)}
  \mathbb P_{\bbf}\!\left[
    L_n(\tilde{\bbf},\bbf)
    -\inf_{\bw\in \mathcal{W}_{\mathrm{C}}}
      L_n(\hat{\bbf}_{\bw|\mathcal M},\bbf)
    \ge\frac{c_*\kappa_F^2\lambda_A^2}{2}
           \sigma^2\psi_{n,M}^{\mathrm C}
  \right]
  \ge p_*q_0(F,A).
\]
Taking the infimum over all estimators proves
\eqref{eq:probability-lower} with \(c_1=c_*\kappa_F^2\lambda_A^2/2\) and \(c_2=p_*q_0(F,A)\).

\medskip
\noindent\textsc{\underline{Step 5: Risk lower bound.}}
Integrating \eqref{eq:packing-testing} gives, for every
\(\tilde{\bbf}\),
\begin{equation}\label{eq:expected-fano}
  \frac1{|\Theta|}
  \sum_{\boldsymbol\theta\in\Theta}
  R_n(\tilde{\bbf},\bbf_{\boldsymbol\theta})
  \ge
  c_*p_*\kappa_F^2\lambda_A^2\sigma^2
  \psi_{n,M}^{\mathrm C}.
\end{equation}
Combining \eqref{eq:expected-fano} with \eqref{eq:oracle-risk},
\begin{align*}
  &\frac1{|\Theta|}
  \sum_{\boldsymbol\theta\in\Theta}
  \left\{
    R_n(\tilde{\bbf},\bbf_{\boldsymbol\theta})
    -\inf_{\bw\in\mathcal W_{\mathrm C}}
      R_n(\hat{\bbf}_{\bw|\mathcal M},\bbf_{\boldsymbol\theta})
  \right\}\ge
  c_*p_*\kappa_F^2\lambda_A^2\sigma^2
  \psi_{n,M}^{\mathrm C}
  -(1+\lambda_A^2)\frac{\sigma^2}{n}.
\end{align*}
Since \(n\psi_{n,M}^{\mathrm C}\) can be made sufficiently large uniformly over the stated range by choosing
\(n_0(F,A)\) and \(M_0(F,A)\) sufficiently large, the last term is absorbed into the first.  The supremum is at least the average over
the finite set \(\{\bbf_{\boldsymbol\theta}:\boldsymbol\theta\in\Theta\}\).
Taking the infimum over \(\tilde{\bbf}\) proves
\eqref{eq:risk-lower}.

\end{proof}

\section{Proof of Theorem~\ref{theo:upper_convex_1}}

\subsection{Preliminaries}\label{sec:concentration_lemmas}

\begin{lemma}[Linear and quadratic sub-Gaussian forms]
\label{lem:subgaussian}
Suppose that $\bxi = (\xi_1,\ldots, \xi_n)^{\top}$ consists of $n$ i.i.d. mean zero sub-Gaussian random variables with $\|\xi_i\|_{\psi_2}\leq \kappa\sigma$. Then, there is a universal constant \(C>0\) such that, for every \(\bz\in\R^n\), every symmetric
\(\bS\in\R^{n\times n}\), and every \(x\geq0\),
\begin{align}
\mathbb P\left(
|\bz^\top\bxi|>
C\kappa\sigma\|\bz\|\sqrt{x}
\right)
&\leq2\exp(-x),
\label{eq:sg_linear}\\
\mathbb P\left(
\left|\bxi^\top\bS\bxi-\sigma^2\tr(\bS)\right|
>
C\kappa^2\sigma^2
\{\|\bS\|_{\F}\sqrt{x}+\|\bS\|_{\op}x\}
\right)
&\leq2\exp(-x).
\label{eq:sg_quadratic}
\end{align}
Consequently, after changing \(C\),
\begin{align}
\mathbb P\Big(
\big|
\bxi^\top\bS\bxi-\sigma^2\tr(\bS)+\bz^\top\bxi
\big|
&>
C\kappa^2
\big[
\{\sigma^2\|\bS\|_{\F}+\sigma\|\bz\|\}\sqrt{x}
\nonumber\\[-1mm]
&\hspace{35mm}
+\sigma^2\|\bS\|_{\op}x
\big]
\Big)
\leq2\exp(-x).
\label{eq:sg_combined}
\end{align}
\end{lemma}

\begin{proof}[Proof of Lemma~\ref{lem:subgaussian}]
We sketch the proof of this lemma briefly. The moment-generating-function characterization of a sub-Gaussian random variable and independence give
\[
\mathbb E\exp(t\bz^\top\bxi)
\leq
\exp(C\kappa^2\sigma^2t^2\|\bz\|^2),
\qquad t\in\R.
\]
Chernoff's method applied to \(t\) and \(-t\) proves
\eqref{eq:sg_linear}. Equation~\eqref{eq:sg_quadratic} is the Hanson--Wright inequality \citep{RudelsonVershynin2013}, since
$
\mathbb E(\bxi^\top\bS\bxi)
=\sum_{i=1}^n S_{ii}\mathbb E\xi_i^2
=\sigma^2\tr(\bS).
$
Finally, combining \eqref{eq:sg_linear} and \eqref{eq:sg_quadratic} gives \eqref{eq:sg_combined}.
\end{proof}

\begin{lemma}[Finite-dimensional mixed-tail chaining]
\label{lem:chaining}
  Let $E$ be a real vector space with $\operatorname{dim}(E) \leq d$, where $d \geq 1$, and let $T \subset E$ be compact. Fix $t_0 \in T$. Let $p_1$ and $p_2$ be seminorms on $E$, and define $r_1 \triangleq \sup _{t \in T} p_1(t-t_0)$ and $r_2 \triangleq \sup _{t \in T} p_2(t-t_0)$. Suppose that the process $\{X_t: t \in T\}$ has almost surely continuous sample paths and satisfies the increment inequality 
  \begin{equation}\label{eq:increment}
    \mathbb{P}\left(\left|X_t-X_s\right|>\kappa\left[\sqrt{u} p_2(t-s)+u p_1(t-s)\right]\right) \leq 2 \exp(-u), 
  \end{equation}
  for every $s, t \in T$ and $u \geq 0$, where $\kappa > 0$ is some constant. Then, there is a universal constant $C$ such that, for every $x \geq 0$, 
  \begin{equation}\label{eq:sup_prob}
    \mathbb{P}\left(\sup _{t \in T}\left|X_t-X_{t_0}\right|>C \kappa\left[r_2 \sqrt{d+x}+r_1(d+x)\right]\right) \leq \exp(-x). 
  \end{equation}
  Consequently, $\mathbb{E} \sup _{t \in T}|X_t-X_{t_0}| \leq C^{\prime} \kappa(r_2 \sqrt{d}+r_1 d)$ for another universal constant $C^{\prime}$. 
\end{lemma}

\begin{proof}[Proof of Lemma~\ref{lem:chaining}]

  \textsc{\underline{Step 1: Introducing a quotient space.}} Without loss of generality, we assume that $r_1, r_2>0$, and define $p(v) \triangleq \max \{p_1(v)/r_1, p_2(v)/r_2 \}$. We see that $p: E \to [0, \infty)$ is also a seminorm, since $p(\lambda v)=|\lambda| p(v)$ and 
  \begin{equation*}
  \begin{split}
     p(v+w) & = \max \left\{\frac{p_1(v+w)}{r_1}, \frac{p_2(v+w)}{r_2} \right\} \\
       & \leq \max \left\{\frac{p_1(v)+p_1(w)}{r_1}, \frac{p_2(v)+p_2(w)}{r_2} \right\} \leq p(v)+p(w). 
  \end{split}
  \end{equation*}
  However, unlike a norm, it is possible to have $p(v)=0$ for a nonzero vector $v$. Define the kernel of $p$ as $K \triangleq \{v \in E: p(v)=0\}$, which contains the directions that are ``invisible'' under the seminorm $p$. In the quotient space $E / K$, define the equivalence class of $v$ as $[v] = Q(v) \triangleq v + K$, where $Q$ denotes the quotient map. On the quotient space, define the norm $\bar{p}([v])=p(v)$. Introducing the quotient space and its norm is necessary, because the unit ball under $p$ may be unbounded along $K$, making a direct volume argument below impossible. After quotienting out the invisible directions in $K$, $\bar{p}$ is a norm and its unit ball is bounded.
  
  \textsc{\underline{Step 2: The covering argument.}} For $\varepsilon>0$, the covering number $N(T, p, \varepsilon)$ is defined as the smallest integer $m$ for which there exist $t_1, \ldots, t_m \in T$ such that $T \subseteq \cup_{j=1}^m\{t \in E: p(t-t_j) \leq \varepsilon \}$. Since $p(t-s)=\bar{p}(Q(t)-Q(s))$, we have $N(T, p, \varepsilon)=N(Q(T), \bar{p}, \varepsilon)$. To upper bound $N(T, p, \varepsilon)$, it suffices to upper bound $N(Q(T), \bar{p}, \varepsilon)$. Let $q\triangleq\operatorname{dim}(E / K) \leq d$, and let $B\triangleq\{z \in E / K: \bar{p}(z) \leq 1\}$ be the unit ball in the quotient space. Choose a maximal $\varepsilon$-separated set $\{z_1, \ldots, z_m\} \subseteq Q(T)$, meaning that $\bar{p}(z_i-z_j)>\varepsilon$ for any $i \neq j$, and no additional point of $Q(T)$ can be added while preserving this property. Maximality implies that $\{z_1, \ldots, z_m\}$ is an $\varepsilon$-net. Hence $N(Q(T), \bar{p}, \varepsilon) \leq m$. Now consider the open balls $z_i+\frac{\varepsilon}{2} B^{\circ}, i=1, \ldots, m$, which are pairwise disjoint. It is also observed that the $\bar{p}$-diameter of $Q(T)$ is at most $2$. Therefore, fixing $z_0 \in Q(T)$, $Q(T) \subseteq z_0+2 B$. It follows that for any $1 \leq i \leq m$, $z_i+\frac{\varepsilon}{2} B^{\circ} \subseteq z_0+(2+\frac{\varepsilon}{2}) B$. Let $\operatorname{vol}_q$ denote Lebesgue volume on the $q$-dimensional quotient space. Translation invariance and scaling give $\operatorname{vol}_q(c B)=c^q \operatorname{vol}_q(B)$. Because the balls $z_i+\frac{\varepsilon}{2} B^{\circ}$ are disjoint and contained in the larger ball $z_0+(2+\frac{\varepsilon}{2}) B$, 
  \begin{equation*}
    m\left(\frac{\varepsilon}{2}\right)^q \operatorname{vol}_q(B) \leq\left(2+\frac{\varepsilon}{2}\right)^q \operatorname{vol}_q(B) .
  \end{equation*}
  Canceling the positive factor $\operatorname{vol}_q(B)$, we obtain 
  \begin{equation}\label{eq:cover_bound}
    N(T, p, \varepsilon) = N(Q(T), \bar{p}, \varepsilon) \leq m \leq\left(\frac{2+\varepsilon / 2}{\varepsilon / 2}\right)^q=\left(1+\frac{4}{\varepsilon}\right)^q \leq\left(1+\frac{4}{\varepsilon}\right)^d .
  \end{equation}
  
  \textsc{\underline{Step 3: Construction of the multiscale nets.}} We first assume that $T$ is finite. Set $T_0 \triangleq \{t_0\}$. Because $p(t-t_0) \leq 1$, this is a $1$-net of $T$. For every $k \geq 1$, choose a $2^{-k}$-net $T_k \subseteq T$ satisfying 
  \begin{equation}\label{eq:16}
    \left|T_k\right| \leq N\big(T, p, 2^{-k}\big) \leq\big(1+4 \cdot 2^k\big)^d, 
  \end{equation}
  where the last inequality follows from (\ref{eq:cover_bound}). For every $t \in T$, choose $\pi_k(t) \in T_k$ such that $p(t-\pi_k(t)) \leq 2^{-k}$. We set $\pi_0(t)=t_0$. The distance between two consecutive approximations satisfies 
  \begin{equation}\label{eq:18}
    \begin{split}
       p\left(\pi_k(t)-\pi_{k-1}(t)\right) & \leq p\left(\pi_k(t)-t\right)+p\left(t-\pi_{k-1}(t)\right) \\
        & \leq 2^{-k}+2^{-(k-1)}  =3 \cdot 2^{-k} .
    \end{split}
  \end{equation}
  Combining (\ref{eq:18}) with the definition of $p$, we have 
  \begin{equation}\label{eq:19}
    p_1\left(\pi_k(t)-\pi_{k-1}(t)\right) \leq 3 r_1 2^{-k}, 
  \end{equation}
  and
  \begin{equation}\label{eq:20}
    p_2\left(\pi_k(t)-\pi_{k-1}(t)\right) \leq 3 r_2 2^{-k}. 
  \end{equation}
  
  Let $\mathcal{E}_k \triangleq \{(\pi_k(t), \pi_{k-1}(t)): t \in T\}$ be the set of distinct edges between levels $k-1$ and $k$. Since the first endpoint lies in $T_k$ and the second lies in $T_{k-1}$, we have $|\mathcal{E}_k| \leq|T_k| \cdot |T_{k-1}|$. Moreover, 
  \begin{equation}\label{eq:23}
    \begin{split}
       \log \left|\mathcal{E}_k\right|  & \leq \log \left|T_k\right|+\log \left|T_{k-1}\right| \leq 6dk, 
    \end{split}
  \end{equation}
  where the last inequality follows from (\ref{eq:16}). 
  
  \textsc{\underline{Step 4: The union bound over all levels.}} Fix $x \geq 0$ and set $u_k\triangleq x+8 d k$. For a fixed edge $(v, w) \in \mathcal{E}_k$, the tail inequality (\ref{eq:increment}), together with (\ref{eq:19})-(\ref{eq:20}), gives 
  \begin{equation}\label{eq:25}
    \mathbb{P}\left(\left|X_v-X_w\right|>3 \kappa 2^{-k}\left\{r_2 \sqrt{u_k}+r_1 u_k\right\}\right) \leq 2 \exp(-u_k). 
  \end{equation}
  Define the event $\mathcal{B}_k$ as the inequality in (\ref{eq:25}) fails for at least one edge in $\mathcal{E}_k$. A union bound over the edges at level $k$ gives 
  \begin{equation*}\label{eq:27}
    \mathbb{P}(\mathcal{B}_k) \leq 2|\mathcal{E}_k| \exp(-u_k) \leq 2 \exp (6 d k) \exp (-x-8 d k)=2 \exp(-x) \exp(-2 d k), 
  \end{equation*}
  where the second inequality follows from (\ref{eq:23}) and the definition of $u_k$. Now taking the union bound over all levels, we have 
  \begin{equation}\label{eq:28}
    \begin{split}
       \mathbb{P}\left(\bigcup_{k \geq 1} \mathcal{B}_k\right) & \leq \sum_{k \geq 1} \mathbb{P}\left(\mathcal{B}_k\right)  \leq 2 \exp(-x) \sum_{k \geq 1} \exp(-2 d k) \\
        & \leq 2 \exp(-x) \sum_{k \geq 1} \exp(-2 k) =\frac{2}{e^2-1} \exp(-x)\\
        &  <\exp(-x) .
    \end{split}
  \end{equation}
  Thus, with probability at least $1-\exp(-x)$, the event in (\ref{eq:25}) holds simultaneously for every level $k$ and every edge in $\mathcal{E}_k$. 
  
  \textsc{\underline{Step 5: Telescoping the chain.}} Because $T$ is finite, after identifying points whose $p$-distance is zero, there exists a sufficiently large $K$ such that we may take $\pi_K(t)=t$ for every $t \in T$. On the event $(\cup_{k = 1}^K \mathcal{B}_k)^c$, we have 
  \begin{equation}\label{eq:30}
    \begin{split}
       \left|X_t-X_{t_0}\right| & \leq 3 \kappa \sum_{k=1}^K 2^{-k}\left[r_2 \sqrt{x+8 d k}+r_1(x+8 d k)\right] \\
         & \leq 3 \kappa r_2 \sum_{k \geq 1} 2^{-k} \sqrt{x+8 d k} +3 \kappa r_1 \sum_{k \geq 1} 2^{-k}(x+8 d k). 
    \end{split}
  \end{equation}
  For the first term in (\ref{eq:30}), we have 
  \begin{equation}\label{eq:33}
    \begin{split}
       \sum_{k \geq 1} 2^{-k} \sqrt{x+8 d k} & \leq \sqrt{x} \sum_{k \geq 1} 2^{-k}+\sqrt{8 d} \sum_{k \geq 1} 2^{-k} \sqrt{k}\\
         & \leq \sqrt{x}  + \sqrt{8 d} \left(\sum_{k \geq 1} 2^{-k}\right)^{1 / 2}\left(\sum_{k \geq 1} k 2^{-k}\right)^{1 / 2}\\
         & = \sqrt{x}+4 \sqrt{d}, 
    \end{split}
  \end{equation}
  where the second inequality follows from the Cauchy–Schwarz inequality, and the last equality uses $\sum_{k \geq 1} k 2^{-k}=2$. For the second term in (\ref{eq:30}), we have 
  \begin{equation}\label{eq:34}
    \sum_{k \geq 1} 2^{-k}(x+8 d k)  =x \sum_{k \geq 1} 2^{-k}+8 d \sum_{k \geq 1} k 2^{-k} =x+16 d
  \end{equation}
  Substituting (\ref{eq:33})-(\ref{eq:34}) into (\ref{eq:30}) yields 
  \begin{equation*}
    \begin{split}
       \sup _{t \in T}\left|X_t-X_{t_0}\right| & \leq 3 \kappa\left[r_2(\sqrt{x}+4 \sqrt{d})+r_1(x+16 d)\right] \leq  3 \kappa \left[r_2 5 \sqrt{d+x} + r_1 16(d+x) \right]\\
       & \leq 48 \kappa\left[r_2 \sqrt{d+x}+r_1(d+x)\right]
    \end{split}
  \end{equation*}
  Together with (\ref{eq:28}), this proves (\ref{eq:sup_prob}) for finite $T$.
  
  \textsc{\underline{Step 6: From finite sets to compact sets.}} Let $D=\{t_1, t_2, \ldots\}$ be a countable dense subset of $T$, and define the increasing finite sets $F_m=\{t_1, \ldots, t_m\}$. Set $Z_m=\max _{t \in F_m}|X_t-X_{t_0}|$. The finite-set result in Step~5 gives, with $B_x=48 \kappa[r_2 \sqrt{d+x}+r_1(d+x)]$, that 
  \begin{equation*}
    \mathbb{P}\left(Z_m>B_x\right) \leq \exp(-x) \quad \text { for every } m .
  \end{equation*}
  Because $F_m \subseteq F_{m+1}$, we have $Z_m \uparrow \sup _{t \in D}|X_t-X_{t_0}|$. Hence 
  \begin{equation*}
    \left\{\sup _{t \in D}\left|X_t-X_{t_0}\right|>B_x\right\}=\bigcup_{m \geq 1}\left\{Z_m>B_x\right\} .
  \end{equation*}
  Since the events on the right are increasing, continuity from below of probability gives 
  \begin{equation}\label{eq:38}
    \mathbb{P}\left(\sup _{t \in D}\left|X_t-X_{t_0}\right|>B_x\right)=\lim _{m \rightarrow \infty} \mathbb{P}\left(Z_m>B_x\right) \leq \exp(-x). 
  \end{equation}
  Finally, on every sample path for which $t \mapsto X_t$ is continuous, density of $D$ implies 
  \begin{equation}\label{eq:39}
    \sup _{t \in T}\left|X_t-X_{t_0}\right|=\sup _{t \in D}\left|X_t-X_{t_0}\right| .
  \end{equation}
  Combining (\ref{eq:38}) and (\ref{eq:39}) proves the result for compact $T$.
  
\end{proof}

\subsection{Proof of the main results in Theorem~\ref{theo:upper_convex_1}}\label{sec:proof_convex_1_core}

For notational simplicity, we write $C_p(\bw)$ for the $C_p$ criterion defined in (\ref{eq:Cp}), suppressing its dependence on $\mathcal{M}$ and $\bY$. In the following lemma, we explore the strong convexity of the Mallows' $C_p$ criterion. This property is essential for deriving a sharp oracle inequality and was also mentioned in Section~4 of \cite{Bellec2018}. However, in deriving the upper bound for the Mallows aggregation estimator in Proposition~7.2 of \cite{Bellec2018}, this property was not further exploited, which led to a suboptimal upper bound in \cite{Bellec2018}. 

\begin{lemma}[Strong convexity of the Mallows criterion]
\label{lem:strong_conv}
  For every \(\bw\in\mathcal{W}_{\mathrm{C}}\),
\begin{equation}\label{eq:strong_cp}
C_p(\hat{\bw})
\leq
C_p(\bw)
-
\|\hat{\bbf}_{\hat{\bw}|\mathcal{M}}
      -\hat{\bbf}_{\bw|\mathcal{M}}\|^2.
\end{equation}
\end{lemma}

\begin{proof}[Proof of Lemma~\ref{lem:strong_conv}]
  Condition on the observed response vector $\bY$, and introduce the $n \times M$ matrix 
\begin{equation*}
  \mathbf{X}(\bY) \triangleq \left[ \bA_1 \bY+\bbb_1, \ldots, \bA_M \bY+\bbb_M \right]. 
\end{equation*}
Then, the aggregated estimator can be written as $\hat{\bbf}_{\bw|\mathcal{M}}=\mathbf{X}(\bY) \bw$. Also let $\boldsymbol{\tau} \triangleq (\tr(\bA_1), \ldots, \tr(\bA_M))^{\top}$. Thus, the Mallows' $C_p$ criterion can be written as 
\begin{equation}\label{eq:C_p_equal}
  C_p(\bw) = \|\mathbf{X}(\bY) \bw-\bY\|^2+2 \sigma^2 \boldsymbol{\tau}^{\top} \bw. 
\end{equation}
Since $\mathcal{W}_{\mathrm{C}}$ is convex, any $\hat{\bw}$ minimizing the differentiable convex function in (\ref{eq:C_p_equal}) over $\mathcal{W}_{\mathrm{C}}$ satisfies $\langle \nabla C_p(\hat{\bw}), \bw-\hat{\bw}\rangle \geq 0$ for $\bw \in \mathcal{W}_{\mathrm{C}}$. The quadratic expansion of $C_p(\bw)$ gives 
\begin{equation}\label{eq:C_p_expan}
  C_p(\bw)= C_p(\hat{\bw})+\langle \nabla C_p(\hat{\bw}), \bw-\hat{\bw}\rangle  +\|\mathbf{X}(\bY)(\bw-\hat{\bw})\|^2.
\end{equation}
Combining (\ref{eq:C_p_equal}) with (\ref{eq:C_p_expan}) proves Lemma~\ref{lem:strong_conv}. 
  
\end{proof}

We are now in a position to prove Theorem~\ref{theo:upper_convex_1}. Since the proof is rather lengthy, we divide it into several steps. 

\begin{proof}[Proof of Theorem~\ref{theo:upper_convex_1}]
\textsc{\underline{Step 1: Basic inequality and representation.}} For a fixed comparator $\bw \in \mathcal{W}_{\mathrm{C}}$, set $\bu=\hat{\bw}-\bw$, and define $\bA_{\bu}\triangleq\sum_{j=1}^M u_j \bA_j$ and $\bbb_{\bu} \triangleq \sum_{j=1}^M u_j \bbb_j $. Then, define $\bV_{\bu} \triangleq \hat{\bbf}_{\hat{\bw}|\mathcal{M}}-\hat{\bbf}_{\bw|\mathcal{M}}=\bA_{\bu} \bY+\bbb_{\bu}$. Based on Lemma~\ref{lem:strong_conv}, we have 
\begin{equation}\label{eq:3.8}
2\big\langle\hat{\bbf}_{\bw|\mathcal{M}}-\bY, \bV_{\bu}\big\rangle \leq- 2\left\|\bV_{\bu}\right\|^2-2 \sigma^2 \tr\left(\bA_{\bu}\right). 
\end{equation}
Therefore, we have 
\begin{equation}\label{eq:3.9}
  \begin{split}
     \big\|\hat{\bbf}_{\hat{\bw}|\mathcal{M}}-\bbf\big\|^2-\big\|\hat{\bbf}_{\bw|\mathcal{M}}-\bbf\big\|^2 & =2\big\langle\hat{\bbf}_{\bw|\mathcal{M}}-\bbf, \bV_{\bu}\big\rangle+\left\|\bV_{\bu}\right\|^2 \\
     & =2\big\langle\hat{\bbf}_{\bw|\mathcal{M}}-\bY, \bV_{\bu}\big\rangle+2\big\langle\bxi, \bV_{\bu}\big\rangle+\left\|\bV_{\bu}\right\|^2\\
     & \leq 2 \bxi^{\top} \bV_{\bu}-2 \sigma^2 \tr\left(\bA_\bu\right)-\left\|\bV_\bu\right\|^2, 
  \end{split}
\end{equation}
where the last inequality follows from (\ref{eq:3.8}). 

Let $\mathcal{D}\triangleq \mathcal{W}_{\mathrm{C}}-\mathcal{W}_{\mathrm{C}}=\{\bw-\bv: \bw, \bv \in \mathcal{W}_{\mathrm{C}}\}$. Define the optimal weight vector minimizing the loss as $\bw^{\circ} \in \underset{\bw \in \mathcal{W}_{\mathrm{C}}}{\arg \min }\|\hat{\bbf}_{\bw|\mathcal{M}}-\bbf\|^2$. Then, $\bu=\hat{\bw}-\bw^{\circ} \in \mathcal{D}$. Therefore, using (\ref{eq:3.9}), we have the upper bound
\begin{equation}\label{eq:3.10}
  \big\|\hat{\bbf}_{\hat{\bw}|\mathcal{M}}-\bbf\big\|^2-\inf_{\bw \in \mathcal{W}_{\mathrm{C}}}\big\|\hat{\bbf}_{\bw|\mathcal{M}}-\bbf\big\|^2 \leq \sup _{\bu \in \mathcal{D}} Z(\bu),
\end{equation}
where $Z(\bu)\triangleq 2 \bxi^{\top} \bV_{\bu}-2 \sigma^2 \tr(\bA_\bu)-\|\bV_\bu\|^2$. Define $\ba_\bu\triangleq\bA_\bu \bbf+\bbb_\bu$. Since $\bY=\bbf+\bxi$, we have $\bV_\bu=\ba_\bu+\bA_\bu \bxi $. Thus, $Z(\bu)$ can be equivalently expressed as 
\begin{equation*}
  \begin{split}
     Z(\bu)= & 2 \bxi^{\top}\left(\ba_\bu+\bA_\bu \bxi\right)-2 \sigma^2 \tr\left(\bA_\bu\right)-\left\|\ba_\bu+\bA_\bu \bxi\right\|^2 \\
= & 2 \bxi^{\top} \ba_\bu+2 \bxi^{\top} \bA_\bu \bxi-2 \sigma^2 \tr\left(\bA_\bu\right)  -\left\|\ba_\bu\right\|^2-2 \bxi^{\top} \bA_\bu^{\top} \ba_\bu-\bxi^{\top} \bA_\bu^{\top} \bA_\bu \bxi .
  \end{split}
\end{equation*}

Introduce the symmetric matrix $\bS_\bu\triangleq\bA_\bu+\bA_\bu^{\top}-\bA_\bu^{\top} \bA_\bu$ and the vector $\bc_\bu\triangleq2\left(\bI_n-\bA_\bu^{\top}\right) \ba_\bu $. Then $2 \bxi^{\top} \bA_\bu \bxi-\bxi^{\top} \bA_\bu^{\top} \bA_\bu \bxi=\bxi^{\top} \bS_\bu \bxi$, and $2 \bxi^{\top} \ba_\bu-2 \bxi^{\top} \bA_\bu^{\top} \ba_\bu=\bxi^{\top} \bc_\bu $. Moreover, $\tr\left(\bS_\bu\right)=2 \tr\left(\bA_\bu\right)-\left\|\bA_\bu\right\|_{\F}^2 $. Hence $\bxi^{\top} \bS_\bu \bxi-2 \sigma^2 \tr\left(\bA_\bu\right)=\left\{\bxi^{\top} \bS_\bu \bxi-\sigma^2 \tr\left(\bS_\bu\right)\right\}-\sigma^2\left\|\bA_\bu\right\|_{\F}^2 $. Define the seminorm $\rho(\bu)^2=\left\|\ba_\bu\right\|^2+\sigma^2\left\|\bA_\bu\right\|_{\F}^2$ and the centered process
\begin{equation}\label{eq:3.17}
W(\bu)=\bxi^{\top} \bS_\bu \bxi-\sigma^2 \tr\left(\bS_\bu\right)+\bxi^{\top} \bc_\bu .
\end{equation}
Then, $Z(\bu)=W(\bu)-\rho(\bu)^2 $. The map $\bu \mapsto\left(\ba_\bu, \sigma \bA_\bu\right)$ is linear, so $\rho$ is indeed a seminorm.

\textsc{\underline{Step 2: Localized tail bound.}} For $r>0$, define the localized set $\mathcal{D}_r\triangleq\{\bu \in \mathcal{D}: \rho(\bu) \leq r\}$ under the seminorm $\rho$. In this step, we prove that for every $x \geq 0$,
\begin{equation}\label{eq:3.20}
\mathbb{P}\left[\sup _{\bu \in \mathcal{D}_r}|W(\bu)|>C_0(\kappa,A)\left\{\sigma r \sqrt{M+x}+\sigma^2(M+x)\right\}\right] \leq \exp(-x),  
\end{equation}
where $C_0(\kappa,A)$ is a constant depending only on $(\kappa,A)$. 

The key point in this step is to use a \emph{control seminorm} that simultaneously controls $\rho(\bu)$ and $\|\bu\|_1$. To that end, define $p_r(\bh)=\max \{\frac{\rho(\bh)}{r}, \frac{\|\bh\|_1}{2}\}$ on the ambient linear space containing $\mathcal{D}$. For $\bu \in \mathcal{D}$, we have $\|\bu\|_1 \leq 2$. Therefore, for any $\bu \in \mathcal{D}_r$, $p_r(\bu) \leq 1$. Let $\bu, \bv \in \mathcal{D}_r$, and write $\bh=\bu-\bv$. By the definition of $p_r$, we see that $\rho(\bh) \leq r p_r(\bh)$ and $\|\bh\|_1 \leq 2 p_r(\bh)$. Consequently, recalling the definition $\rho(\bh)=\sqrt{\|\ba_\bh\|^2+\sigma^2\|\bA_\bh\|_{\F}^2}$, we have $\left\|\ba_\bh\right\| \leq r p_r(\bh)$ and $\left\|\bA_\bh\right\|_{\F} \leq \frac{r}{\sigma} p_r(\bh)$ for any $\bh=\bu-\bv$. In addition, using (\ref{eq:operator-class-2}), we see that 
$
\left\|\bA_\bh\right\|_{\op} = \big\| \sum_{j=1}^M h_j \bA_j \big\|_{\op} \leq A\|\bh\|_1 \leq 2 A p_r(\bh)$. Also, since $\bu, \bv \in \mathcal{D}_r \subseteq \mathcal{D}$, we have $\left\|\bA_\bu\right\|_{\op} \leq 2 A$ and $\left\|\bA_\bv\right\|_{\op} \leq 2 A$, and because $\bu, \bv \in \mathcal{D}_r$, we have $\left\|\ba_\bu\right\| \leq r$ and $\left\|\ba_\bv\right\| \leq r$. 

Now, the increment of the centered process in (\ref{eq:3.17}) can be written as 
\begin{equation}\label{eq:3.17_1}
\begin{split}
W(\bu)-W(\bv)= & \bxi^{\top}\left(\bS_\bu-\bS_\bv\right) \bxi-\sigma^2 \tr\left(\bS_\bu-\bS_\bv\right)  +\bxi^{\top}\left(\bc_\bu-\bc_\bv\right).
\end{split}
\end{equation}
For the matrix increment in (\ref{eq:3.17_1}), observe that
\begin{equation}\label{eq:mat_incr}
  \bS_\bu-\bS_\bv  =\bA_\bh+\bA_\bh^{\top}-\big(\bA_\bu^{\top} \bA_\bu-\bA_\bv^{\top} \bA_\bv\big) =\bA_\bh+\bA_\bh^{\top}-\bA_\bu^{\top} \bA_\bh-\bA_\bh^{\top} \bA_\bv . 
\end{equation}
Thus, based on (\ref{eq:mat_incr}), the Frobenius norm of $\bS_\bu-\bS_\bv$ is upper bounded by 
\begin{equation}\label{eq:mat_incr_1}
  \begin{split}
     \sigma^2\left\|\bS_\bu-\bS_\bv\right\|_{\F} & \leq\sigma^2 \big(2+\left\|\bA_\bu\right\|_{\op}+\left\|\bA_\bv\right\|_{\op}\big)\left\|\bA_\bh\right\|_{\F}   \leq\sigma^2(2+4 A)\left\|\bA_\bh\right\|_{\F} \leq C_A \sigma r p_r(\bh). 
  \end{split}
\end{equation}
Similarly, using (\ref{eq:mat_incr}), the operator norm of $\bS_\bu-\bS_\bv$ is upper bounded by 
\begin{equation}\label{eq:mat_incr_2}
  \sigma^2\left\|\bS_\bu-\bS_\bv\right\|_{\op} \leq C_A \sigma^2 p_r(\bh) .
\end{equation}
For the linear part in (\ref{eq:3.17_1}), we have 
\begin{equation*}
  \left(\bc_\bu-\bc_\bv\right)  =2\big(\bI_n-\bA_\bu^{\top}\big) \ba_\bu-2\big(\bI_n-\bA_\bv^{\top}\big) \ba_\bv  =2\big(\bI_n-\bA_\bu^{\top}\big) \ba_\bh-2\bA_\bh^{\top} \ba_\bv. 
\end{equation*}
Therefore, 
\begin{equation}\label{eq:a_h_bound}
  \sigma\left\|\bc_\bu-\bc_\bv\right\|  \leq 2\sigma\big(1+\left\|\bA_\bu\right\|_{\op}\big)\left\|\ba_\bh\right\|+2\sigma\left\|\bA_\bh\right\|_{\op}\left\|\ba_\bv\right\| 
\leq C_A\sigma r p_r(\bh),
\end{equation}

Combining Lemma~\ref{lem:subgaussian} with the upper bounds in (\ref{eq:mat_incr_1}), (\ref{eq:mat_incr_2}), and (\ref{eq:a_h_bound}), we see that 
\begin{equation*}
\begin{split}
& \mathbb{P}\left(|W(\bu)-W(\bv)|>C_{\kappa,A} p_r(\bu-\bv)\left\{\sigma r \sqrt{z}+\sigma^2 z\right\}\right)\\
& \leq \mathbb{P}\left(|W(\bu)-W(\bv)|>C\kappa^2\left[\left\{\sigma^2\left\|\bS_\bu-\bS_\bv\right\|_{\F}+\sigma\left\|\bc_\bu-\bc_\bv\right\|\right\} \sqrt{z}+\sigma^2\left\|\bS_\bu-\bS_\bv\right\|_{\op} z\right]\right)\\
& \leq 2 \exp(-z). 
\end{split}
\end{equation*}
The set $\mathcal{D}_r$ is contained in a vector space of dimension at most $M$, and $p_r(\bu) \leq 1$ on $\mathcal{D}_r$. Moreover, $W(\boldsymbol{0})=0$, and $\bu \mapsto W(\bu)$ is continuous. The finite-dimensional mixed-tail chaining result in Lemma~\ref{lem:chaining} therefore gives (\ref{eq:3.20}). 

\textsc{\underline{Step 3: Peeling argument.}} We now prove
\begin{equation}\label{eq:3.36}
\sup _{\bu \in \mathcal{D}}\left\{W(\bu)-\rho(\bu)^2\right\} \leq C_{\kappa,A} \sigma^2(M+x) 
\end{equation}
with probability at least $1-\exp(-x)$. We decompose $\mathcal{D}$ into dyadic annuli. Set $q=M+x$. Since $M \geq 1$, we have $q \geq 1$. Let $\Gamma \geq 1$ be a sufficiently large constant depending only on $(\kappa,A)$, to be specified below, and define $r_k=\Gamma 2^k \sigma \sqrt{q}$ for $k=0,1,2, \ldots$. Define the inner region $\mathcal{A}_0=\mathcal{D}_{r_0}$, and for $k \geq 1$, define the annuli
\begin{equation}\label{eq:3.39}
\mathcal{A}_k=\left\{\bu \in \mathcal{D}: r_{k-1}<\rho(\bu) \leq r_k\right\}. 
\end{equation}
Since $r_k \to \infty$, we have $\mathcal{D}=\cup_{k=0}^{\infty} \mathcal{A}_k$. For every $k \geq 0$, choose the confidence parameter $x_k=x+2(k+1)$. Apply (\ref{eq:3.20}) to $\mathcal{D}_{r_k}$ with $x$ replaced by $x_k$. Let $\mathcal{E}_k$ denote the resulting event:
\begin{equation}\label{eq:3.42}
\sup _{\bu \in \mathcal{D}_{r_k}} W(\bu) \leq C_0\left\{\sigma r_k \sqrt{M+x_k}+\sigma^2\left(M+x_k\right)\right\} .
\end{equation}
Then, $\mathbb{P}(\mathcal{E}_k^c) \leq \exp(-x_k) $. Taking a union bound, we see that 
\begin{equation}\label{eq:3.43}
  \begin{split}
     \mathbb{P}\left(\bigcup_{k \geq 0} \mathcal{E}_k^c\right) & \leq \sum_{k \geq 0} \exp(-x_k)  =\exp(-x) \sum_{k \geq 0} \exp[-2(k+1)]\\
& =\frac{\exp(-x)}{e^2-1}  <\exp(-x) .
  \end{split}
\end{equation}
Thus, with probability at least $1-\exp(-x)$, all the localized inequalities (\ref{eq:3.42}) hold simultaneously. 

Write $q_k\triangleq M+x_k=q+2 k+2$. It is easy to verify that $\sqrt{q_k} \leq 2 \sqrt{q} 2^{k / 2}$. In the following, we divide the proof into two regions. In the inner region $\mathcal{A}_0$, for any $\bu \in \mathcal{A}_0$, we have $Z(\bu)=W(\bu)-\rho(\bu)^2 \leq W(\bu)$. Using (\ref{eq:3.42}) with $k=0$, together with $q_0 \leq 4 q$, we have 
$$
\begin{aligned}
\sup _{\bu \in \mathcal{A}_0} Z(\bu) & \leq C_0\left\{\sigma r_0 \sqrt{q_0}+\sigma^2 q_0\right\}  \leq C_0\left\{\sigma(\Gamma \sigma \sqrt{q})(2 \sqrt{q})+4 \sigma^2 q\right\} \\
& =C_0(2 \Gamma+4) \sigma^2 q \leq C_{\kappa,A} \sigma^2(M+x) .
\end{aligned}
$$

For the outer annuli $\mathcal{A}_k$ with $k \geq 1$, if $\bu \in \mathcal{A}_k$, then $-\rho(\bu)^2 \leq-r_{k-1}^2$. Moreover, $\mathcal{A}_k \subseteq \mathcal{D}_{r_k}$, so (\ref{eq:3.42}) gives
\begin{equation}\label{eq:ssccs}
  \begin{aligned}
\sup _{\bu \in \mathcal{A}_k} Z(\bu) & =\sup _{\bu \in \mathcal{A}_k}\left\{W(\bu)-\rho(\bu)^2\right\} \leq C_0\left\{\sigma r_k \sqrt{q_k}+\sigma^2 q_k\right\}-r_{k-1}^2 .
\end{aligned}
\end{equation}
We now evaluate the three terms in (\ref{eq:ssccs}) separately. For the first term, we see that 
$$
\begin{aligned}
\sigma r_k \sqrt{q_k} & \leq \sigma\left(\Gamma 2^k \sigma \sqrt{q}\right)\left(2 \sqrt{q} 2^{k / 2}\right)  =2 \Gamma \sigma^2 q 2^{3 k / 2} .
\end{aligned}
$$
For the second term, we have $\sigma^2 q_k \leq 4 \sigma^2 q 2^k$. The third term can be written as $r_{k-1}^2 =\left(\Gamma 2^{k-1} \sigma \sqrt{q}\right)^2 =\frac{\Gamma^2}{4} \sigma^2 q 2^{2 k}$. Substituting these quantities into (\ref{eq:ssccs}) gives,
$$
\sup _{\bu \in \mathcal{A}_k} Z(\bu) \leq \sigma^2 q\left\{2 C_0 \Gamma 2^{3 k / 2}+4 C_0 2^k-\frac{\Gamma^2}{4} 2^{2 k}\right\} \leq \sigma^2 q 2^{2 k}\left\{\sqrt{2} C_0 \Gamma+2 C_0-\frac{\Gamma^2}{4}\right\}.
$$
Choose $\Gamma=\Gamma(\kappa,A)$ sufficiently large that $\frac{\Gamma^2}{4} \geq \sqrt{2} C_0 \Gamma+2 C_0 $. Then, for every $k \geq 1$, $\sup _{\bu \in \mathcal{A}_k} Z(\bu) \leq 0$. Combining the above results gives 
\begin{equation}\label{eq:3.58}
  \begin{split}
\sup _{\bu \in \mathcal{D}} Z(\bu) & =\max \left\{\sup _{\bu \in \mathcal{A}_0} Z(\bu), \sup _{k \geq 1} \sup _{\bu \in \mathcal{A}_k} Z(\bu)\right\}  \leq C_{\kappa,A} \sigma^2(M+x) 
\end{split}
\end{equation}
with probability at least $1-\exp(-x)$. This proves (\ref{eq:3.36}).

\textsc{\underline{Step 4: Completion of the oracle inequality.}} Combining (\ref{eq:3.10}) and (\ref{eq:3.58}), with probability at least $1 - \exp(-x)$, 
\begin{equation}
L_n(\hat{\bbf}_{\hat{\bw}|\mathcal{M}},\bbf)-\inf _{\bw \in \mathcal{W}_{\mathrm{C}}}L_n(\hat{\bbf}_{\bw|\mathcal{M}},\bbf) \leq C_{\kappa,A} \frac{\sigma^2(M+x)}{n} .
\end{equation}
This proves (\ref{eq:upper_small_prob}). Integrating the exponential tail gives
$$
\mathbb{E}L_n(\hat{\bbf}_{\hat{\bw}|\mathcal{M}},\bbf)-\mathbb{E}\inf _{\bw \in \mathcal{W}_{\mathrm{C}}}L_n(\hat{\bbf}_{\bw|\mathcal{M}},\bbf) \leq C_{\kappa,A}\frac{\sigma^2 M}{n} .
$$
This proves (\ref{eq:upper_small_random}). Combining this inequality with the fact that $\mathbb{E}\inf _{\bw \in \mathcal{W}_{\mathrm{C}}}L_n(\hat{\bbf}_{\bw|\mathcal{M}},\bbf) \leq \inf _{\bw \in \mathcal{W}_{\mathrm{C}}}\mathbb{E}L_n(\hat{\bbf}_{\bw|\mathcal{M}},\bbf)=\inf _{\bw \in \mathcal{W}_{\mathrm{C}}}R_n(\hat{\bbf}_{\bw|\mathcal{M}},\bbf)$ gives 
$$
\mathbb{E}L_n(\hat{\bbf}_{\hat{\bw}|\mathcal{M}},\bbf) \leq \inf _{\bw \in \mathcal{W}_{\mathrm{C}}}R_n(\hat{\bbf}_{\bw|\mathcal{M}},\bbf)+C_{\kappa,A} \frac{\sigma^2 M}{n}.
$$
This proves (\ref{eq:upper_small_risk}).

\end{proof}

\section{Proof of Theorem~\ref{theo:upper_convex_2}}

\subsection{Preliminaries}

Before proving Theorem~\ref{theo:upper_convex_2}, we first collect some notation that will be used throughout the proof. For \(\bw\in\mathcal{W}_{\mathrm C}\), write
$
\mathcal L(\bw)
\triangleq
\|\hat{\bbf}_{\bw|\mathcal{M}}-\bbf\|^2.
$
Then, $L_n(\hat{\bbf}_{\bw|\mathcal{M}},\bbf)=n^{-1}\mathcal L(\bw)$. We also use the shorthand
$
\hat{\bbf}_j
\triangleq
\hat{\bbf}_{\be_j|\mathcal{M}}
=
\bA_j\bY+\bbb_j.
$
For \(\bu\in\R^M\), define
$
\bA_{\bu}=\sum_{j=1}^M u_j\bA_j$,
$\bbb_{\bu}=\sum_{j=1}^M u_j\bbb_j$, and
$\bV_{\bu}=\bA_{\bu}\bY+\bbb_{\bu}$.
For notational simplicity, we use $C_p(\bw)$ to denote the $C_p$ criterion defined in (\ref{eq:Cp}), suppressing its dependence on $\mathcal{M}$ and $\bY$. In addition, define
\begin{equation}\label{eq:G_def}
G(\bu)
\triangleq
2\bxi^\top\bV_{\bu}
-2\sigma^2\tr(\bA_{\bu}).
\end{equation}
For every \(\bs,\bt\in\mathcal{W}_{\mathrm C}\), a direct expansion of the $C_p$ criterion yields
\begin{equation}\label{eq:loss_Cp_identity}
\mathcal L(\bs)-\mathcal L(\bt)
=
C_p(\bs)-C_p(\bt)+G(\bs-\bt).
\end{equation}

We next present two useful lemmas. Lemma~\ref{lem:segment} establishes a one-dimensional segment inequality. 

\begin{lemma}[Uniform segment inequality]\label{lem:segment}
Let
\(\mathcal S=[\bw_0,\bw_1]\subseteq\mathcal{W}_{\mathrm C}\) be any deterministic line segment. Then, for every \(x\geq0\),
\begin{equation}\label{eq:segment}
\mathbb P\left[
\sup_{\bu\in\mathcal S-\mathcal S}
\left\{
G(\bu)-\frac12\|\bV_{\bu}\|^2
\right\}
>
C_{\kappa,A}\sigma^2(1+x)
\right]
\leq \exp(-x).
\end{equation}
\end{lemma}

\begin{proof}[Proof of Lemma~\ref{lem:segment}]
Regard the two endpoints as a new dictionary consisting of the affine estimators
$
\hat{\bbf}_{\bw_k|\mathcal{M}}
=
\bA_{\bw_k}\bY+\bbb_{\bw_k}$, where
$k=0,1.
$
Since \(\bw_k\in\mathcal{W}_{\mathrm C}\) and $\|\bA_j\|_{\op}\leq A$, we have
\[
\|\bA_{\bw_k}\|_{\op}
\leq
\sum_{j=1}^M (w_k)_j\|\bA_j\|_{\op}
\leq A,
\]
where $(w_k)_j$ denotes the $j$th component of $\bw_k$. Moreover,
$
\mathcal S-\mathcal S
=
\{t(\bw_1-\bw_0):|t|\leq1\}.
$
Under the linear map from the two-dimensional coefficient vector of this new dictionary to the original coefficient space, its difference set is identified with \(\mathcal S-\mathcal S\), and the associated processes coincide with those defined above.
Apply the global process bound \eqref{eq:3.58}, rather than merely the final oracle inequality of Theorem~\ref{theo:upper_convex_1}, to this two-element dictionary. It gives
\begin{equation}\label{eq:segment_first}
\mathbb P\left[
\sup_{\bu\in\mathcal S-\mathcal S}
\{W(\bu)-\rho(\bu)^2\}
>
C_{\kappa,A}\sigma^2(1+x)
\right]
\leq \exp(-x),
\end{equation}
after absorbing the fixed dictionary size \(2\) into the constant.

For completeness, the process in \eqref{eq:segment_first} is exactly the process needed in Lemma~\ref{lem:segment}. Indeed, with
\(\ba_{\bu}=\bA_{\bu}\bbf+\bbb_{\bu}\),
$
\bV_{\bu}=\ba_{\bu}+\bA_{\bu}\bxi,
$
the definitions in the paragraph above equation (\ref{eq:3.17}) give
\begin{align*}
W(\bu)-\rho(\bu)^2
&=
\bxi^\top
(\bA_{\bu}+\bA_{\bu}^\top-\bA_{\bu}^\top\bA_{\bu})
\bxi
-2\sigma^2\tr(\bA_{\bu})
+2\bxi^\top(\bI_n-\bA_{\bu}^\top)\ba_{\bu}
-\|\ba_{\bu}\|^2\\
&=
2\bxi^\top(\ba_{\bu}+\bA_{\bu}\bxi)
-2\sigma^2\tr(\bA_{\bu})
-\|\ba_{\bu}+\bA_{\bu}\bxi\|^2\\
&=
G(\bu)-\|\bV_{\bu}\|^2.
\end{align*}
Thus, \eqref{eq:segment_first} controls
\(G(\bu)-\|\bV_{\bu}\|^2\).

Finally, \(\mathcal S-\mathcal S\) is symmetric and star-shaped. Hence
\(\bu/2\in\mathcal S-\mathcal S\), and linearity gives
$
G(\bu/2)=\frac12G(\bu)$ and
$
\bV_{\bu/2}=\frac12\bV_{\bu}.
$
Therefore,
\[
G(\bu)-\frac12\|\bV_{\bu}\|^2
=
2\left\{
G(\bu/2)-\|\bV_{\bu/2}\|^2
\right\}.
\]
Applying \eqref{eq:segment_first} to \(\bu/2\) and enlarging the constant proves \eqref{eq:segment}.
\end{proof}

\begin{lemma}[Maurey grid]\label{lem:maurey_grid}
For an integer \(1\leq m\leq M\), let
\begin{equation}\label{eq:G_m}
\mathcal G_m
\triangleq
\left\{
\frac1m\sum_{\ell=1}^m\be_{j_\ell}:
j_1,\ldots,j_m\in\{1,\ldots,M\}
\right\}
\subseteq\mathcal{W}_{\mathrm C}.
\end{equation}
Then
\begin{equation}\label{eq:grid_cardinality}
N_m\triangleq|\mathcal G_m|
=
\binom{M+m-1}{m}
\leq
\left(\frac{2eM}{m}\right)^m.
\end{equation}
\end{lemma}

\begin{proof}[Proof of Lemma~\ref{lem:maurey_grid}]
Every grid point is determined by nonnegative integers
\(k_1,\ldots,k_M\) summing to \(m\), so the stars-and-bars identity gives the equality. Moreover,
\[
\binom{M+m-1}{m}
\leq
\left\{\frac{e(M+m-1)}m\right\}^m
\leq
\left(\frac{2eM}{m}\right)^m,
\]
because \(m\leq M\).
\end{proof}

\subsection{Proof of the main results in Theorem~\ref{theo:upper_convex_2}}\label{sec:proof_convex_2_core}

The proof below uses Maurey sampling twice. The first application approximates the realized convex oracle by a point of \(\mathcal G_m\). The second sparsifies the data-dependent Mallows weight vector \(\hat{\bw}\). Because the first grid point is itself data-dependent, the empirical-process bound must be uniform over every pair of points in \(\mathcal G_m\). This additional uniformization is the key step that yields a deviation inequality for the Mallows aggregate. Set
\begin{equation}\label{eq:Lambda_def}
\Lambda^2
\triangleq
\sigma^2+F^2+B^2.
\end{equation}

\begin{proof}[Proof of Theorem~\ref{theo:upper_convex_2}]

\noindent\textsc{\underline{Step 1: First Maurey sampling: approximation of the realized oracle.}}
For each realization of \(\bY\), choose
$
\bw^\circ
\in
\argmin_{\bw\in\mathcal{W}_{\mathrm C}}\mathcal L(\bw),
$
using a fixed tie-breaking rule. On an auxiliary probability space, conditionally on \(\bY\), let
\(J_1^\circ,\ldots,J_m^\circ\) be independent with
$
\mathbb P_\Theta(J_\ell^\circ=j\mid\bY)=w_j^\circ,
$
and set
\[
\tilde{\bw}^{\,\circ}
=
\frac1m\sum_{\ell=1}^m\be_{J_\ell^\circ}
\in\mathcal G_m.
\]
Conditional on \(\bY\), the variance identity gives
\begin{equation}\label{eq:first_maurey_identity}
\begin{split}
\mathbb E_\Theta\left[\mathcal L(\tilde{\bw}^{\,\circ})\mid\bY\right]
&=
\mathcal L(\bw^\circ)
+
\frac1m
\sum_{j=1}^M
w_j^\circ
\|\hat{\bbf}_j-\hat{\bbf}_{\bw^\circ|\mathcal{M}}\|^2\\
&=
\mathcal L(\bw^\circ)
+
\frac1m
\left\{
\sum_{j=1}^M w_j^\circ\|\hat{\bbf}_j\|^2
-\|\hat{\bbf}_{\bw^\circ|\mathcal{M}}\|^2
\right\}.
\end{split}
\end{equation}
Define the maximal candidate energy by
$
\mathcal H(\bY)
\triangleq
\max_{1\leq j\leq M}\|\hat{\bbf}_j\|^2.
$
It follows from \eqref{eq:first_maurey_identity} that
\[
\mathbb E_\Theta\left[\mathcal L(\tilde{\bw}^{\,\circ})\mid\bY\right]
\leq
\mathcal L(\bw^\circ)+\frac{\mathcal H(\bY)}m.
\]
Let
$
\bw_1
\in
\argmin_{\bg\in\mathcal G_m}\mathcal L(\bg),
$
again using a fixed tie-breaking rule. Since \(\tilde{\bw}^{\,\circ}\) takes values in the finite set \(\mathcal G_m\), the minimum over that set cannot exceed its conditional average. Consequently, pointwise in \(\bY\),
\begin{equation}\label{eq:pathwise_grid_oracle}
\mathcal L(\bw_1)
\leq
\inf_{\bw\in\mathcal{W}_{\mathrm C}}\mathcal L(\bw)
+\frac{\mathcal H(\bY)}m.
\end{equation}
Notice that \(\bw_1\) is generally data-dependent. 

\medskip
\noindent\textsc{\underline{Step 2: Uniform control over all grid-to-grid segments.}}
For every \(\bg_0,\bg_1\in\mathcal G_m\), let
$
\mathcal S_{\bg_0,\bg_1}=[\bg_0,\bg_1]
$
and define
\begin{equation}\label{eq:Gamma_m}
\Gamma_m(\bY)
=
\max_{\bg_0,\bg_1\in\mathcal G_m}
\sup_{\bu\in\mathcal S_{\bg_0,\bg_1}-\mathcal S_{\bg_0,\bg_1}}
\left\{
G(\bu)-\frac12\|\bV_{\bu}\|^2
\right\}.
\end{equation}
Since \(\boldsymbol{0}\in\mathcal S_{\bg_0,\bg_1}-\mathcal S_{\bg_0,\bg_1}\),
\(\Gamma_m(\bY)\geq0\). Apply Lemma~\ref{lem:segment} with confidence parameter
\(x+2\log N_m+\log2\) and take a union bound over at most \(N_m^2\) segments. After enlarging the constant, we obtain
\begin{equation}\label{eq:Gamma_tail}
\mathbb P\left[
\Gamma_m(\bY)
>
C_{\kappa,A}\sigma^2
\{1+2\log N_m+x\}
\right]
\leq \frac12\exp(-x).
\end{equation}

\medskip
\noindent\textsc{\underline{Step 3: Second Maurey sampling: sparsifying \(\hat{\bw}\).}}
Condition on \(\bY\). Let
\(J_1,\ldots,J_m\) be conditionally independent with
$
\mathbb P_\Theta(J_\ell=j\mid\bY)
=
\hat w_j,
$
and define
$
\tilde{\bw}
=
\frac1m\sum_{\ell=1}^m\be_{J_\ell}
\in\mathcal G_m.
$
Let
\begin{equation}\label{eq:conditional_variance}
\mathcal V(\bY)
=
\sum_{j=1}^M
\hat w_j
\|\hat{\bbf}_j-\hat{\bbf}_{\hat{\bw}|\mathcal{M}}\|^2.
\end{equation}
For each realization \(\bg\) of \(\tilde{\bw}\), take
$
\bv_{\bg}
\in
\argmin_{\bv\in[\bw_1,\bg]}
\mathcal L(\bv).
$
The first-order optimality condition for the convex quadratic
\(\mathcal L\) on this segment gives
\begin{equation}\label{eq:loss_strong_segment}
\mathcal L(\bg)-\mathcal L(\bv_{\bg})
\geq
\|\hat{\bbf}_{\bg|\mathcal{M}}
-\hat{\bbf}_{\bv_{\bg}|\mathcal{M}}\|^2
=
\|\bV_{\bg-\bv_{\bg}}\|^2.
\end{equation}
By \eqref{eq:loss_Cp_identity}, \eqref{eq:Gamma_m}, and
\eqref{eq:loss_strong_segment}, we have
\begin{align*}
\mathcal L(\bg)-\mathcal L(\bv_{\bg})
&=
C_p(\bg)-C_p(\bv_{\bg})
+G(\bg-\bv_{\bg})\\
&\leq
C_p(\bg)-C_p(\bv_{\bg})
+
\frac12
\{\mathcal L(\bg)-\mathcal L(\bv_{\bg})\}
+\Gamma_m(\bY).
\end{align*}
Thus,
\begin{equation}\label{eq:g_v_bound}
\mathcal L(\bg)-\mathcal L(\bv_{\bg})
\leq
2\{C_p(\bg)-C_p(\bv_{\bg})\}
+2\Gamma_m(\bY).
\end{equation}
Because \(\bw_1\in[\bw_1,\bg]\),
\(\mathcal L(\bv_{\bg})\leq\mathcal L(\bw_1)\). Also,
\(\hat{\bw}\) minimizes \(C_p\) over \(\mathcal{W}_{\mathrm C}\), so
\(C_p(\hat{\bw})\leq C_p(\bv_{\bg})\). Hence
\begin{equation}\label{eq:g_bound}
\mathcal L(\bg)
\leq
\mathcal L(\bw_1)
+2\{C_p(\bg)-C_p(\hat{\bw})\}
+2\Gamma_m(\bY).
\end{equation}

Conditional on \(\bY\), the Maurey variance identities are
\begin{align}
\mathbb E_\Theta\left[\mathcal L(\tilde{\bw})\mid\bY\right]
&=
\mathcal L(\hat{\bw})
+\frac{\mathcal V(\bY)}m,
\label{eq:second_maurey_loss}\\
\mathbb E_\Theta\left[C_p(\tilde{\bw})\mid\bY\right]
&=
C_p(\hat{\bw})
+\frac{\mathcal V(\bY)}m.
\label{eq:second_maurey_cp}
\end{align}
For \eqref{eq:second_maurey_cp}, the squared-residual term has the same variance decomposition as the loss, whereas
\(\tr(\bA_{\bw})\) is linear in \(\bw\). Taking the conditional expectation of \eqref{eq:g_bound} and using
\eqref{eq:second_maurey_loss}--\eqref{eq:second_maurey_cp} gives
\begin{equation}\label{eq:central_maurey}
\mathcal L(\hat{\bw})
\leq
\mathcal L(\bw_1)
+\frac{\mathcal V(\bY)}m
+2\Gamma_m(\bY).
\end{equation}
Moreover, the variance identity gives
\begin{equation}\label{eq:V_energy}
\mathcal V(\bY)
=
\sum_{j=1}^M\hat w_j\|\hat{\bbf}_j\|^2
-\|\hat{\bbf}_{\hat{\bw}|\mathcal{M}}\|^2
\leq
\mathcal H(\bY).
\end{equation}
Combining \eqref{eq:pathwise_grid_oracle}, \eqref{eq:central_maurey}, and \eqref{eq:V_energy}, we obtain the pathwise master inequality
\begin{equation}\label{eq:pathwise_master}
\mathcal L(\hat{\bw})
-
\inf_{\bw\in\mathcal{W}_{\mathrm C}}\mathcal L(\bw)
\leq
\frac{2\mathcal H(\bY)}m
+2\Gamma_m(\bY).
\end{equation}

\medskip
\noindent\textsc{\underline{Step 4: High-probability control of the maximal candidate energy.}}
For every \(1\leq j\leq M\), the conditions
\(\bbf\in\mathcal{F}_n(F)\) and
\(\mathcal{M}\in\mathbb{M}_{M,A,B}\) give
\begin{align*}
\|\hat{\bbf}_j\|
&\leq
\|\bA_j\bY\|+\|\bbb_j\|\leq
A\|\bxi\|+(AF+B)\sqrt n.
\end{align*}
Consequently,
\begin{equation}\label{eq:H_deterministic_bound}
\mathcal H(\bY)
\leq
2A^2\|\bxi\|^2
+2n(AF+B)^2.
\end{equation}
Apply Lemma~\ref{lem:subgaussian} with \(\bS=\bI_n\) and confidence parameter
\(x+\log4\). Since
\(\sqrt{n(x+\log4)}\leq\{n+x+\log4\}/2\), after enlarging the constant we obtain
\begin{equation}\label{eq:noise_energy_tail}
\mathbb P\left[
\|\bxi\|^2
>
C_\kappa\sigma^2(n+x)
\right]
\leq
\frac12\exp(-x).
\end{equation}
It follows from \eqref{eq:H_deterministic_bound}--\eqref{eq:noise_energy_tail} that
\begin{equation}\label{eq:H_tail}
\mathbb P\left[
\mathcal H(\bY)
>
C_{\kappa,A}\left\{n\Lambda^2+\sigma^2x\right\}
\right]
\leq
\frac12\exp(-x).
\end{equation}
Importantly, no union bound over \(j\) is needed: the common bound
\(\|\bA_j\bxi\|\leq A\|\bxi\|\) holds simultaneously for all candidates.

Intersect the events in \eqref{eq:Gamma_tail} and \eqref{eq:H_tail}. By a union bound, their intersection has probability at least \(1-\exp(-x)\). On this event, \eqref{eq:pathwise_master} and Lemma~\ref{lem:maurey_grid} imply
\begin{align}
L_n(\hat{\bbf}_{\hat{\bw}|\mathcal{M}},\bbf)
-
\inf_{\bw\in\mathcal{W}_{\mathrm C}}
L_n(\hat{\bbf}_{\bw|\mathcal{M}},\bbf)
&\leq
C_{\kappa,A}\left[
\frac{\Lambda^2}{m}
+\frac{\sigma^2x}{mn}
+\frac{\sigma^2\{1+2\log N_m+x\}}{n}
\right]
\nonumber\\
&\quad\leq
C_{\kappa,A}\left[
\Lambda^2
\left\{
\frac1m
+\frac mn\log\left(\frac{2eM}{m}\right)
\right\}
+\frac{\sigma^2x}{n}
\right].
\label{eq:maurey_general_probability}
\end{align}
For the last inequality, we used
\(\log N_m\leq m\log(2eM/m)\),
\(\sigma^2x/(mn)\leq\sigma^2x/n\), and
\(1/n\leq(m/n)\log(2eM/m)\). We have therefore proved the sharp deviation inequality
\begin{equation}\label{eq:maurey_general_tail}
\mathbb P_{\bbf}\left[
L_n(\hat{\bbf}_{\hat{\bw}|\mathcal{M}},\bbf)
-
\inf_{\bw\in\mathcal{W}_{\mathrm C}}
L_n(\hat{\bbf}_{\bw|\mathcal{M}},\bbf)
>
C_{\kappa,A}\left[
\Lambda^2
\left\{
\frac1m
+\frac mn\log\left(\frac{2eM}{m}\right)
\right\}
+\frac{\sigma^2x}{n}
\right]
\right]
\leq
\exp(-x).
\end{equation}

\medskip
\noindent\textsc{\underline{Step 5: The bound in expectation.}}
The excess loss on the left-hand side of \eqref{eq:maurey_general_tail} is nonnegative. Integrating its tail and absorbing the resulting term \(1/n\) into
\((m/n)\log(2eM/m)\), we obtain
\begin{align}
\mathbb E L_n(\hat{\bbf}_{\hat{\bw}|\mathcal{M}},\bbf)
&\leq
\mathbb E\inf_{\bw\in\mathcal{W}_{\mathrm C}}
L_n(\hat{\bbf}_{\bw|\mathcal{M}},\bbf)
+
C_{\kappa,A}\Lambda^2
\left\{
\frac1m
+\frac mn\log\left(\frac{2eM}{m}\right)
\right\}
\nonumber\\
&\leq
\inf_{\bw\in\mathcal{W}_{\mathrm C}}
R_n(\hat{\bbf}_{\bw|\mathcal{M}},\bbf)
+
C_{\kappa,A}\Lambda^2
\left\{
\frac1m
+\frac mn\log\left(\frac{2eM}{m}\right)
\right\}.
\label{eq:maurey_general}
\end{align}
Here the second inequality follows from
$
\mathbb E\inf_{\bw\in\mathcal{W}_{\mathrm C}}
L_n(\hat{\bbf}_{\bw|\mathcal{M}},\bbf)
\leq
\inf_{\bw\in\mathcal{W}_{\mathrm C}}
R_n(\hat{\bbf}_{\bw|\mathcal{M}},\bbf).
$

\medskip
\noindent\textsc{\underline{Step 6: Choice of \(m\).}}
Suppose \(M>\sqrt n\) and set
$
\ell_M=\log(eM/\sqrt n).
$
Then \(1\leq\ell_M\leq n\). Choose
\[
m=\left\lceil\sqrt{ n/\ell_M}\right\rceil.
\]
This choice satisfies \(1\leq m\leq M\).
Since \(m\geq\sqrt{n/\ell_M}\) and
\(m\leq2\sqrt{n/\ell_M}\),
$
\frac1m\leq\sqrt{\frac{\ell_M}{n}},
$
and
\begin{align*}
\log\left(\frac{2eM}{m}\right)
&\leq
\log\left(\frac{2eM\sqrt{\ell_M}}{\sqrt n}\right)=
\ell_M+\log2+\frac12\log\ell_M
\leq C\ell_M.
\end{align*}
Consequently, we see that 
\[
\frac mn\log\left(\frac{2eM}{m}\right)
\leq
C\sqrt{\frac{\ell_M}{n}}.
\]
Substitution into \eqref{eq:maurey_general_tail} gives, for every \(x\geq0\),
\begin{equation}\label{eq:maurey_rate_probability}
\mathbb P_{\bbf}\left[
L_n(\hat{\bbf}_{\hat{\bw}|\mathcal{M}},\bbf)
-
\inf_{\bw\in\mathcal{W}_{\mathrm C}}
L_n(\hat{\bbf}_{\bw|\mathcal{M}},\bbf)
>
C_{\kappa,A}\left\{
\Lambda^2
\sqrt{\frac{\log(eM/\sqrt n)}n}
+\frac{\sigma^2x}{n}
\right\}
\right]
\leq
\exp(-x).
\end{equation}
Similarly, substitution into \eqref{eq:maurey_general} gives
\begin{equation}\label{eq:maurey_rate}
\mathbb E L_n(\hat{\bbf}_{\hat{\bw}|\mathcal{M}},\bbf)
\leq
\inf_{\bw\in\mathcal{W}_{\mathrm C}}
R_n(\hat{\bbf}_{\bw|\mathcal{M}},\bbf)
+
C_{\kappa,A}\Lambda^2
\sqrt{\frac{\log(eM/\sqrt n)}n}.
\end{equation}
This completes the proof.
\end{proof}

\section{Proof of Theorem~\ref{theo:diffuse-linear-lower}}

We first give the affine construction for $s\leq n$. Put
$d=D-1$ and fix orthonormal vectors
$\be_1,\ldots,\be_d\in\R^n$. Let
\begin{equation}\label{eq:stable-lower-scales}
  \mathbf Q=\frac12\bI_n,
  \qquad
  \beta=\frac{B\sqrt n}{2}.
\end{equation}
Define the first $D$ candidate pairs by
\begin{equation}\label{eq:stable-hard-collection}
  (\bA_j,\bbb_j)
  =
  (\mathbf Q,\beta\be_j),
  \quad j=1,\ldots,d,
  \qquad
  (\bA_D,\bbb_D)
  =
  (\mathbf Q,\boldsymbol{0}),
\end{equation}
and repeat $(\mathbf Q,\boldsymbol{0})$ if $M>D$. These pairs satisfy
\begin{equation}\label{eq:stable-hard-bounds}
  \max_{1\leq j\leq M}\|\bA_j\|_{\op}
  =\frac12\leq A,
  \qquad
  \max_{1\leq j\leq M}\|\bbb_j\|_n
  =\frac B2\leq B.
\end{equation}
They also span a $D$-dimensional pair space. Indeed, if a linear combination
of the first $D$ pairs is zero, its intercept component first forces the
first $d$ coefficients to be zero, after which its matrix component forces
the last coefficient to be zero. Furthermore,
\begin{equation}\label{eq:stable-hard-stable-rank}
  \mathcal A(\mathcal M_0)
  =\operatorname{span}\{\mathbf Q\},
  \qquad
  s_{\mathcal A}(\mathcal M_0)
  =\operatorname{sr}(\mathbf Q)
  =n.
\end{equation}
Thus
$\mathcal M_0\in\mathbb M_{M,D,A,B}^{s}$ whenever
$s\leq n$.

Let $\Omega_d\subseteq\{-1,1\}^d$ be a Varshamov--Gilbert packing such that
\begin{equation}\label{eq:stable-hypercube-packing}
  \log|\Omega_d|\geq c_0d,
  \qquad
  d_{\mathrm H}(\boldsymbol\omega,\boldsymbol\omega')
  \geq\frac d8
\end{equation}
for all distinct
$\boldsymbol\omega,\boldsymbol\omega'\in\Omega_d$. Choose a sufficiently
small universal constant $\alpha>0$ and set
\begin{equation}\label{eq:stable-hard-amplitude}
  \delta
  =
  \alpha
  \left(
    \sigma\wedge F\sqrt{\frac nd}
  \right),
  \qquad
  \bbf_{\boldsymbol\omega}
  =
  \delta\sum_{j=1}^d\omega_j\be_j.
\end{equation}
Then
\begin{equation}\label{eq:stable-hard-in-class}
  \|\bbf_{\boldsymbol\omega}\|_n^2
  =
  \frac{\delta^2d}{n}
  =
  \alpha^2
  \left\{
    F^2\wedge\frac{\sigma^2d}{n}
  \right\}
  \leq F^2,
\end{equation}
so every hard mean belongs to $\mathcal F_n(F)$.
For each $\boldsymbol\omega\in\Omega_d$, define
\begin{equation}\label{eq:stable-oracle-weights}
  w_j(\boldsymbol\omega)
  =
  \frac{\delta\omega_j}{\beta},
  \quad j=1,\ldots,d,
  \qquad
  w_D(\boldsymbol\omega)
  =
  -\sum_{j=1}^dw_j(\boldsymbol\omega),
\end{equation}
and set the remaining weights equal to zero. The weights sum to zero, and
hence the common random matrix component cancels:
\begin{equation}\label{eq:stable-exact-oracle}
  \bA_{\bw(\boldsymbol\omega)}
  =\boldsymbol{0},
  \qquad
  \bbb_{\bw(\boldsymbol\omega)}
  =\bbf_{\boldsymbol\omega},
  \qquad
  \hat{\bbf}_{\bw(\boldsymbol\omega)|\mathcal M_0}
  =\bbf_{\boldsymbol\omega}.
\end{equation}
Consequently, both the realized and risk oracles are zero on the hard family.

For distinct
$\boldsymbol\omega,\boldsymbol\omega'\in\Omega_d$, we see that 
\begin{align}
  \|\bbf_{\boldsymbol\omega}
    -\bbf_{\boldsymbol\omega'}\|^2
  &=
  4\delta^2
  d_{\mathrm H}(\boldsymbol\omega,\boldsymbol\omega')
  \geq
  \frac{\delta^2d}{2},
  \label{eq:stable-hard-separation}\\
  \operatorname{KL}\!\left(
    P_{\bbf_{\boldsymbol\omega}},
    P_{\bbf_{\boldsymbol\omega'}}
  \right)
  &=
  \frac{
    \|\bbf_{\boldsymbol\omega}
      -\bbf_{\boldsymbol\omega'}\|^2
  }{2\sigma^2}
  \leq
  2\alpha^2d.
  \label{eq:stable-hard-kl}
\end{align}
Choosing $\alpha$ sufficiently small and applying Fano's lemma with the
nearest-neighbor reduction gives a universal constant $p_0>0$ such that
\begin{align}
  \inf_{\tilde{\bbf}}
  \frac{1}{|\Omega_d|}
  \sum_{\boldsymbol\omega\in\Omega_d}
  \mathbb P_{\bbf_{\boldsymbol\omega}}\!\left[
    L_n(\tilde{\bbf},\bbf_{\boldsymbol\omega})
    \geq
    \frac{\delta^2d}{8n}
  \right]
  &\geq p_0.
  \label{eq:stable-fano-probability}
\end{align}
Indeed, whenever
$\|\tilde{\bbf}-\bbf_{\boldsymbol\omega}\|^2<\delta^2d/8$, the nearest
element of the packing is $\bbf_{\boldsymbol\omega}$, by
\eqref{eq:stable-hard-separation}. Fano's lemma controls the probability of
failure of this nearest-neighbor decoder because, by
\eqref{eq:stable-hard-kl}, the pairwise Kullback--Leibler divergences are a
sufficiently small multiple of $\log|\Omega_d|$. Since a supremum is at
least an average, \eqref{eq:stable-fano-probability} implies
\begin{align}
  \inf_{\tilde{\bbf}}
  \sup_{\boldsymbol\omega\in\Omega_d}
  \mathbb P_{\bbf_{\boldsymbol\omega}}\!\left[
    L_n(\tilde{\bbf},\bbf_{\boldsymbol\omega})
    \geq
    c
    \left\{
      F^2\wedge\frac{\sigma^2d}{n}
    \right\}
  \right]
  &\geq p_0.
  \label{eq:stable-fano-supremum}
\end{align}
Moreover, integrating the probability bound yields
\begin{align}
  \inf_{\tilde{\bbf}}
  \sup_{\boldsymbol\omega\in\Omega_d}
  R_n(\tilde{\bbf},\bbf_{\boldsymbol\omega})
  &\geq
  c
  \left\{
    F^2\wedge\frac{\sigma^2d}{n}
  \right\}.
  \label{eq:stable-fano-risk}
\end{align}
Since $d=D-1\geq D/2$, the right-hand sides of
\eqref{eq:stable-fano-supremum}--\eqref{eq:stable-fano-risk} are equivalent,
up to universal constants, to
$F^2\wedge\sigma^2D/n$. Combining these inequalities with
\eqref{eq:stable-exact-oracle} proves
\eqref{eq:diffuse-lower-prob}--\eqref{eq:diffuse-lower-risk} for
$s\leq n$.

For $s>n$, take $\bA_j=\boldsymbol{0}$ and
$\bbb_j=\beta\be_j$ for $j=1,\ldots,D$, repeating the last pair if
$M>D$. Then $D(\mathcal M_0)=D$ and
$s_{\mathcal A}(\mathcal M_0)=+\infty$. Let
$\Omega_D\subseteq\{-1,1\}^D$ be the corresponding hypercube packing and
put
\begin{equation*}
  \delta_D
  =
  \alpha
  \left(
    \sigma\wedge F\sqrt{\frac nD}
  \right),
  \qquad
  \bbf_{\boldsymbol\omega}
  =
  \delta_D\sum_{j=1}^D\omega_j\be_j.
\end{equation*}
For every $\boldsymbol\omega\in\Omega_D$, the weights
$w_j(\boldsymbol\omega)=\delta_D\omega_j/\beta$, $1\leq j\leq D$,
produce the exact aggregate
$\hat{\bbf}_{\bw(\boldsymbol\omega)|\mathcal M_0}
=\bbf_{\boldsymbol\omega}$. Repeating the preceding Fano argument with
$d$ replaced by $D$ proves both stated lower bounds. This deterministic
subcase also handles $D=1$.

\section{Proof of Theorem~\ref{theo:stable-upper}}

Choose a deterministic basis
$
  (\mathbf C_1,\bc_1),\ldots,(\mathbf C_D,\bc_D)
$
of the space $\mathcal V(\mathcal M)$. For $\bu\in\R^D$, define
\begin{equation}\label{eq:basis-linear-maps}
  \mathbf C_{\bu}
  =\sum_{k=1}^Du_k\mathbf C_k,
  \qquad
  \bc_{\bu}
  =\sum_{k=1}^Du_k\bc_k,
  \qquad
  \ba_{\bu}
  =\mathbf C_{\bu}\bbf+\bc_{\bu}.
\end{equation}
Since the basis $
  (\mathbf C_1,\bc_1),\ldots,(\mathbf C_D,\bc_D)
$ spans $\mathcal V(\mathcal M)$, we have 
\begin{equation}\label{eq:basis-original-equivalence}
  \left\{
    (\mathbf C_{\bu},\bc_{\bu}):\bu\in\R^D
  \right\}
  =
  \left\{
    (\bA_{\bw},\bbb_{\bw}):\bw\in\R^M
  \right\}.
\end{equation}
Thus the basis is only a nonredundant coordinate system; it does not change
the attainable aggregates, the Mallows criterion values, or either oracle in
Theorem~\ref{theo:stable-upper}.

We now introduce the random quantities needed in the proof:
\begin{equation}\label{eq:random-design-proof}
  \bX
  =
  \left[
    \mathbf C_1\bY+\bc_1,\ldots,
    \mathbf C_D\bY+\bc_D
  \right]
  \in\R^{n\times D},
  \quad
  \bG=\bX^\top\bX,
  \quad
  \bt
  =
  \left(
    \tr(\mathbf C_1),\ldots,\tr(\mathbf C_D)
  \right)^\top.
\end{equation}
In these coordinates, the Mallows criterion is
$
  \|\bX\bu-\bY\|^2+2\sigma^2\bt^\top\bu.
$
When $\bG$ is positive definite, its unique minimizer is
\begin{equation}\label{eq:cp-explicit-proof}
  \hat{\bu}^{\mathrm{L}}
  =
  \bG^{-1}
  \left(\bX^\top\bY-\sigma^2\bt\right).
\end{equation}

Define
\begin{equation}\label{eq:rho-linear}
  \rho_{\bbf}(\bu)^2
  \triangleq
  \|\ba_{\bu}\|^2
  +\sigma^2\|\mathbf C_{\bu}\|_{\F}^2.
\end{equation}
Because the pairs in $
  (\mathbf C_1,\bc_1),\ldots,(\mathbf C_D,\bc_D)
$ are linearly independent,
$\rho_{\bbf}$ is a norm. Indeed, $\rho_{\bbf}(\bu)=0$ implies first that
$\mathbf C_{\bu}=\boldsymbol{0}$ and then that
$\bc_{\bu}=\boldsymbol{0}$, so $\bu=\boldsymbol{0}$. Let
$\bH_{\bbf}$ be its positive-definite Gram matrix:
$
  \bu^\top\bH_{\bbf}\bu
  =
  \rho_{\bbf}(\bu)^2.
$
Since
$\bX\bu=\ba_{\bu}+\mathbf C_{\bu}\bxi$, we have 
$
  \bH_{\bbf}=\mathbb E_{\bbf}(\bG).
$
We have for every $\bbf\in\mathcal F_n(F)$,
\begin{equation}\label{eq:projection-properties}
  \left\|
    \hat{\bbf}_{\mathcal M}^{\mathrm{TCp}}-\bbf
  \right\|
  \leq
  \left\|
    \hat{\bbf}_{\mathcal M}^{\mathrm{Cp}}-\bbf
  \right\|,
  \qquad
  \left\|
    \hat{\bbf}_{\mathcal M}^{\mathrm{TCp}}-\bbf
  \right\|^2
  \leq4F^2n.
\end{equation}

\medskip
\noindent\textsc{\underline{Step 1: the exact self-normalized identity.}}
Define the score vector
\begin{equation}\label{eq:score}
  \bq
  \triangleq
  \bX^\top\bxi-\sigma^2\bt.
\end{equation}
On the event $\{\bG\succ\boldsymbol{0}\}$, the realized least-squares
oracle in the basis coordinates is
$
  \bu^*
  =
  \bG^{-1}\bX^\top\bbf.
$
Equations \eqref{eq:cp-explicit-proof} and \eqref{eq:score} give
$
  \hat{\bu}^{\mathrm{L}}-\bu^*
  =
  \bG^{-1}\bq.
$
Because $\bX\bu^*$ is the orthogonal projection of $\bbf$ onto
$\operatorname{col}(\bX)$, the Pythagorean identity yields
\begin{align}
  \|\bX\hat{\bu}^{\mathrm{L}}-\bbf\|^2
  -\inf_{\bu\in\R^D}\|\bX\bu-\bbf\|^2
  &=
  \|\bX(\hat{\bu}^{\mathrm{L}}-\bu^*)\|^2
  =
  \bq^\top\bG^{-1}\bq.
  \label{eq:exact-excess}
\end{align}

\medskip
\noindent\textsc{\underline{Step 2: fixed-direction concentration.}}
For fixed $\bu\in\R^D$, abbreviate
$\mathbf C=\mathbf C_{\bu}$, $\ba=\ba_{\bu}$, and
$\rho=\rho_{\bbf}(\bu)$. From \eqref{eq:score},
\begin{equation}\label{eq:score-expansion}
  \bu^\top\bq
  =
  \ba^\top\bxi
  +\bxi^\top\mathbf C\bxi
  -\sigma^2\tr(\mathbf C).
\end{equation}
For a nonsymmetric $\mathbf C$, the quadratic form in
\eqref{eq:score-expansion} depends only on its symmetric part. The linear
sub-Gaussian inequality and the Hanson--Wright inequality therefore imply
that, for every $z\geq0$,
\begin{align}
\mathbb P\!\left[
  |\bu^\top\bq|
  >
  C_\kappa
  \left\{
    \sigma\|\ba\|\sqrt z
    +\sigma^2\|\mathbf C\|_{\F}\sqrt z
    +\sigma^2\|\mathbf C\|_{\op}z
  \right\}
\right]
&\leq2\exp(-z).
\label{eq:score-hw}
\end{align}
By \eqref{eq:rho-linear}, we have
$
  \|\ba\|\leq\rho$
  and
  $\sigma\|\mathbf C\|_{\F}\leq\rho.
$
If $\mathbf C\neq\boldsymbol{0}$, the minimum-stable-rank condition gives
\begin{equation}\label{eq:stable-rank-control}
  \sigma\|\mathbf C\|_{\op}
  \leq
  \frac{\sigma\|\mathbf C\|_{\F}}{\sqrt s}
  \leq
  \frac{\rho}{\sqrt s}.
\end{equation}
If $\mathbf C=\boldsymbol{0}$, the corresponding operator-norm term
vanishes. Hence
\begin{align}
\mathbb P\!\left[
  |\bu^\top\bq|
  >
  C_\kappa\sigma\rho_{\bbf}(\bu)
  \left\{
    \sqrt z+\frac{z}{\sqrt s}
  \right\}
\right]
&\leq2\exp(-z).
\label{eq:fixed-score}
\end{align}

For the Gram form, we have 
$
  \|\bX\bu\|^2-\rho^2
  =
  2\ba^\top\mathbf C\bxi
  +\bxi^\top\mathbf C^\top\mathbf C\bxi
  -\sigma^2\|\mathbf C\|_{\F}^2.
$
Using the upper bounds
\begin{align}
  \|\mathbf C^\top\ba\|
  \leq
  \|\mathbf C\|_{\op}\|\ba\|,
  \qquad
  \|\mathbf C^\top\mathbf C\|_{\F}
   \leq
  \|\mathbf C\|_{\op}\|\mathbf C\|_{\F},
  \qquad
  \|\mathbf C^\top\mathbf C\|_{\op}
  =
  \|\mathbf C\|_{\op}^2,
\end{align}
and applying the same concentration inequalities gives
\begin{align}
\mathbb P\!\left[
  \left|
    \|\bX\bu\|^2-\rho_{\bbf}(\bu)^2
  \right|
  >
  C_\kappa\rho_{\bbf}(\bu)^2
  \left\{
    \sqrt{\frac zs}+\frac zs
  \right\}
\right]
&\leq4\exp(-z).
\label{eq:fixed-gram}
\end{align}
As before, the terms involving $s$ are zero for directions with
$\mathbf C_{\bu}=\boldsymbol{0}$.

\medskip
\noindent\textsc{\underline{Step 3: uniform score and Gram bounds.}}
A $1/2$-net of the Euclidean unit sphere in $\R^D$ has cardinality at most
$5^D$. Apply \eqref{eq:fixed-score} to
$\bu=\bH_{\bbf}^{-1/2}\bv$ at the net points, for which
$\rho_{\bbf}(\bu)=1$, and use the standard net-to-sphere argument. This
gives, for every $x\geq0$,
\begin{align}
\mathbb P\!\left[
  \|\bH_{\bbf}^{-1/2}\bq\|
  >
  C_\kappa\sigma
  \left\{
    \sqrt{D+x}+\frac{D+x}{\sqrt s}
  \right\}
\right]
&\leq2e^{-x}.
\label{eq:uniform-score}
\end{align}
Integrating this tail shows that, whenever $s\geq D$,
\begin{equation}\label{eq:score-second-moment}
  \mathbb E
  \|\bH_{\bbf}^{-1/2}\bq\|^2
  \leq
  C_\kappa\sigma^2D.
\end{equation}

Similarly, apply \eqref{eq:fixed-gram} on a $1/4$-net of cardinality at most
$9^D$ with $z=c_\kappa s$, where $c_\kappa>0$ is sufficiently small. The
standard net bound for symmetric matrices yields
\begin{align}
\mathbb P\!\left[
  \left\|
    \bH_{\bbf}^{-1/2}
    \bG
    \bH_{\bbf}^{-1/2}
    -\bI_D
  \right\|_{\op}
  >\frac12
\right]
&\leq
2\exp\{C_\kappa D-c_\kappa s\}.
\label{eq:uniform-gram}
\end{align}

\medskip
\noindent\textsc{\underline{Step 4: completion of the probability bound.}}
Let
\begin{equation}\label{eq:good-gram-event}
  \mathcal E_G
  =
  \left\{
    \left\|
      \bH_{\bbf}^{-1/2}
      \bG
      \bH_{\bbf}^{-1/2}
      -\bI_D
    \right\|_{\op}
    \leq\frac12
  \right\}.
\end{equation}
On $\mathcal E_G$, the matrix $\bG$ is positive definite and
\begin{equation}\label{eq:inverse-comparison}
  \bG^{-1}\preceq2\bH_{\bbf}^{-1}.
\end{equation}
Moreover, the aggregate obtained from
$\hat{\bu}^{\mathrm{L}}$ coincides with
$\hat{\bbf}_{\mathcal M}^{\mathrm{Cp}}$. Combining
\eqref{eq:projection-properties}, \eqref{eq:exact-excess}, and
\eqref{eq:inverse-comparison} gives
\begin{align}
  &\left\|
    \hat{\bbf}_{\mathcal M}^{\mathrm{TCp}}-\bbf
  \right\|^2
  -\inf_{\bw\in\R^M}
    \|\hat{\bbf}_{\bw|\mathcal M}-\bbf\|^2
  \leq
  2\|\bH_{\bbf}^{-1/2}\bq\|^2
  \quad\text{on }\mathcal E_G.
  \label{eq:excess-on-good-event}
\end{align}
Equations \eqref{eq:uniform-score} and \eqref{eq:uniform-gram}, together
with $(r+t)^2\leq2r^2+2t^2$, yield the untruncated remainder in
\eqref{eq:stable-upper-prob}. On the other hand,
\eqref{eq:projection-properties} implies deterministically that
\begin{equation}\label{eq:deterministic-excess-cap}
  L_n(\hat{\bbf}_{\mathcal M}^{\mathrm{TCp}},\bbf)
  -\inf_{\bw\in\R^M}
    L_n(\hat{\bbf}_{\bw|\mathcal M},\bbf)
  \leq4F^2.
\end{equation}
Combining these two bounds proves \eqref{eq:stable-upper-prob}. Under
the condition $s\geq C_\kappa(D+x)$, both the quadratic remainder and
the second exceptional probability are absorbed, giving
\eqref{eq:stable-upper-prob-simplified}.

\medskip
\noindent\textsc{\underline{Step 5: completion of the risk bound.}}
Using \eqref{eq:excess-on-good-event} on $\mathcal E_G$ and the second
inequality in \eqref{eq:projection-properties} on $\mathcal E_G^c$, we get
\begin{align}
  \mathbb E_{\bbf}
  \left\|
    \hat{\bbf}_{\mathcal M}^{\mathrm{TCp}}-\bbf
  \right\|^2
  &\leq
  \mathbb E_{\bbf}
  \inf_{\bw\in\R^M}
    \|\hat{\bbf}_{\bw|\mathcal M}-\bbf\|^2
  \nonumber\\
  &\quad+
  2\mathbb E
  \|\bH_{\bbf}^{-1/2}\bq\|^2
  +4F^2n\mathbb P(\mathcal E_G^c).
  \label{eq:expectation-decomposition}
\end{align}
Condition $
  s
  \geq
  C_\kappa
  \left\{
    D+
    \log\!\left(
      e+\frac{nF^2}{\sigma^2D}
    \right)
  \right\},
$ and
\eqref{eq:uniform-gram} imply
\begin{equation}\label{eq:bad-gram-small}
  F^2n\mathbb P(\mathcal E_G^c)
  \leq
  C_\kappa\sigma^2D.
\end{equation}
Indeed, with $r=nF^2/(\sigma^2D)$, the right-hand side of
\eqref{eq:uniform-gram} is at most $C_\kappa/(e+r)$, and
$r/(e+r)\leq1$. The same condition implies $s\geq D$, so
\eqref{eq:score-second-moment} applies. Finally,
\begin{equation}\label{eq:random-to-risk-oracle}
  \mathbb E_{\bbf}
  \inf_{\bw\in\R^M}
    \|\hat{\bbf}_{\bw|\mathcal M}-\bbf\|^2
  \leq
  \inf_{\bw\in\R^M}
  \mathbb E_{\bbf}
    \|\hat{\bbf}_{\bw|\mathcal M}-\bbf\|^2.
\end{equation}
Substituting \eqref{eq:score-second-moment},
\eqref{eq:bad-gram-small}, and \eqref{eq:random-to-risk-oracle} into
\eqref{eq:expectation-decomposition}, and then dividing by $n$, gives an
excess-risk bound of order $\sigma^2D/n$. The deterministic cap
\eqref{eq:deterministic-excess-cap} gives the simultaneous bound of order
$F^2$. Taking the smaller of the two proves \eqref{eq:stable-upper-risk}.

\newpage
\bibliographystyle{apalike}
\bibliography{paper-ref}

\end{sloppypar}
\end{document}